\documentclass[a4paper,12pt,reqno]{amsart}
\usepackage[T1]{fontenc}
\usepackage[utf8]{inputenc}
\usepackage[libertinus,vvarbb]{newtx}
\usepackage[margin=1in]{geometry}
\usepackage{enumitem}
\usepackage{xparse}
\usepackage{xurl}
\usepackage{graphicx}
\usepackage[bookmarksdepth=2]{hyperref}

\numberwithin{equation}{section}

\newtheorem{theorem}{Theorem}[section]
\newtheorem{lemma}[theorem]{Lemma}
\newtheorem{corollary}[theorem]{Corollary}
\newtheorem{proposition}[theorem]{Proposition}

\theoremstyle{definition}
\newtheorem{definition}[theorem]{Definition}

\newtheorem{example}[theorem]{Example}
\newtheorem{remark}[theorem]{Remark}

\DeclareMathOperator{\re}{Re}
\DeclareMathOperator{\im}{Im}
\DeclareMathOperator{\Arg}{Arg}
\DeclareMathOperator{\arccot}{arccot}

\newcommand{\C}{\mathbb{C}}
\newcommand{\R}{\mathbb{R}}
\newcommand{\Z}{\mathbb{Z}}
\newcommand{\N}{\mathbb{N}}
\newcommand{\ph}{\varphi}
\newcommand{\eps}{\varepsilon}
\newcommand{\thet}{\vartheta}
\newcommand{\ro}{\varrho}
\newcommand{\ind}{\mathbb{1}}

\newcommand{\rf}[1]{^{\overline{#1}}}

\renewcommand{\le}{\leqslant}
\renewcommand{\ge}{\geqslant}

\mathchardef\mathhyphen="2D
\newcommand{\cm}{\mathscr{C}\mspace{-4mu}\mathscr{M}}
\newcommand{\am}{\mathscr{A}\mspace{-4mu}\mathscr{M}}
\newcommand{\amcm}{\am\mspace{-2mu}\mathhyphen\mspace{-1mu}\cm}

\NewDocumentCommand{\formula}{ssom}{%
 \IfBooleanTF{#1}{%
  \IfBooleanTF{#2}{%
   \IfValueTF{#3}%
    {\begin{align}\label{#3}\begin{gathered}#4\end{gathered}\end{align}}%
    {\begin{gather}#4\end{gather}}%
  }{%
   \IfValueTF{#3}%
    {\begin{align}\label{#3}\begin{aligned}#4\end{aligned}\end{align}}%
    {\begin{gather*}#4\end{gather*}}%
  }%
 }{%
  \IfValueTF{#3}%
   {\begin{align}\label{#3}#4\end{align}}%
   {\begin{align*}#4\end{align*}}%
 }%
}

\begin{document}

\title[Two-sided bell-shaped sequences]{Two-sided bell-shaped sequences}
\author{Mateusz Kwaśnicki, Jacek Wszoła}
\thanks{Work supported by the Polish National Science Centre (NCN) grant no.\@ 2019/33/B/ST1/03098}
\address{Mateusz Kwaśnicki, Jacek Wszoła \\ \normalfont Department of Analysis and Stochastic Processes \\ Faculty of Pure and Applied Mathematics \\ Wrocław University of Science and Technology \\ ul. Wybrzeże Wyspiańskiego 27 \\ 50-370 Wrocław, Poland}
\email{mateusz.kwasnicki@pwr.edu.pl{\normalfont,} jacek.wszola@pwr.edu.pl}
\date{\today}
\keywords{Bell-shape, Pólya frequency sequence, completely monotone sequence, generating function, Pick function, random walk}
\subjclass[2020]{
 40A05, 
 26A51, 
 39A70, 
 60E10, 
 60E07
}

\begin{abstract}
A nonnegative real function $f$ is \emph{bell-shaped} if it converges to zero at $\pm\infty$ and the $n$th derivative of $f$ changes sign $n$ times for every $n = 0, 1, 2, \ldots$\, Similarly, a nonnegative sequence $(a(k) : k \in \Z)$ is \emph{bell-shaped} if it converges to zero at $\pm\infty$ and the $n$th iterated difference of $a(k)$ changes sign $n$ times for every $n = 0, 1, 2, \ldots$\, A characterisation of bell-shaped functions was given by Thomas Simon and the first named author, and recently a similar result for one-sided bell-shaped sequences was found by the authors. In the present article we give a complete description of two-sided bell-shaped sequences. Our main result proves that bell-shaped sequences are convolutions of \emph{Pólya frequency sequences} and what we call \emph{absolutely monotone-then-completely monotone sequences}, and it provides an equivalent, and relatively easy to verify, condition in terms of holomorphic extensions of the generating function. We also prove that if $f$ is a bell-shaped function, then $f(k)$ is a bell-shaped sequence.
\end{abstract}

\maketitle

%
%

\section{Introduction}
\label{sec:intro}

\subsection{History of bell-shape in a nutshell}

The notion of a bell-shaped function has been present in mathematical literature since 1940s when it was introduced in the context of statistical games (see Section~6.11.C in~\cite{karlin}). A nonnegative, smooth real function $f$ is \textit{bell-shaped} if it converges to zero at $\pm\infty$ and its $n$th derivative $f^{(n)}$ changes sign exactly $n$ times for every $n = 0, 1, 2, \ldots\,$ Many common probability distributions have bell-shaped density functions, including the normal distribution $(2 \pi)^{-1/2} \exp(-x^2/2)$, the Cauchy distribution $\pi^{-1} (1 + x^2)^{-1}$, and the Lévy distribution $(2 \pi)^{-1/2} \exp(1 / (2 x)) \ind_{(0, \infty)}(x)$. In fact, as proved in~\cite{kwasnicki}, all stable distributions have bell-shaped densities; see also~\cite{gawronski,simon}. There are no compactly supported bell-shaped functions~\cite{hirschman}. Density functions of hitting times of 1-D diffusion processes are examples of one-sided (that is, supported in a half-line) bell-shaped functions~\cite{js}. A complete characterisation of the class of bell-shaped functions was given in~\cite{ks}. As a corollary, it follows that probability distributions with bell-shaped density functions are necessarily infinitely divisible.

At this point, a natural question arises: is there a discrete analogue of the theory of bell-shaped functions? In other words, can one prove similar results for appropriately defined bell-shaped sequences? This problem was tackled in the authors' previous work~\cite{kw}, where the following definition was introduced. A nonnegative two-sided sequence $(a(k) : k \in \Z)$ is \textit{bell-shaped} if it converges to zero at $\pm \infty$ and the sequence of its $n$th iterated differences $(\Delta^n a(k))$ changes sign exactly $n$ times for every $n = 0,1, 2, \ldots$

A one-sided sequence $(a(k) : k \in \N)$ (where $\N = \{0, 1, 2, \ldots\}$) can be identified with the corresponding two-sided sequence satisfying $a(k) = 0$ for $k < 0$. The methods used in~\cite{kw} only allowed to characterise one-sided bell-shaped sequences. Theorem~1.1 in~\cite{kw} provides two equivalent conditions for a one-sided sequence to be bell-shaped, in terms of the holomorphic extension of the generating function, or in terms of Pólya frequency sequences and completely monotone sequences. More precisely, the former condition requires that the generating function is the exponential of a Pick function with appropriate boundary values. The latter one asserts that every one-sided bell-shaped sequence is the convolution of a summable Pólya frequency sequence and a completely monotone sequence which converges to zero. As a corollary, all discrete stable distributions (see~\cite{sv}) have one-sided bell-shaped probability mass functions.

The above results are completely analogous to those available for one-sided bell-shaped functions. However, the theory of bell-shaped sequences is not entirely parallel to its continuous counterpart: while there are no compactly supported bell-shaped functions, many finitely supported sequences are bell-shaped. For example, binomial distributions have bell-shaped probability mass functions $\binom{n}{k} p^k (1 - p)^{n - k} \ind_{\{0, 1, \ldots, n\}}(k)$.

The main purpose of the present work is to extend the results of~\cite{kw} and characterise all two-sided bell-shaped sequences; see Theorem~\ref{thm:main}. The major difference, and hence the main difficulty, lies in the fact that for one-sided bell-shaped sequences, the generating function is well defined and holomorphic in the unit disc in the complex plane, and standard inversion formulae apply. On the other hand, for two-sided sequences the generating function is only defined on the unit circle. Furthermore, generating functions of two-sided bell-shaped sequences are no longer exponentials of Pick functions. To overcome these difficulties, we need to develop a novel inversion formula, introduce and study a new class of holomorphic functions, and adjust appropriately the idea of the proof developed in~\cite{kwasnicki,ks,kw}.

In Theorem~\ref{thm:sample} we show that a variety of bell-shaped sequences arise by sampling bell-shaped functions. This direct link between the discrete and continuous notions of bell-shape is rather unexpected: we know no direct proof of this result, and our argument involves the characterisation of bell-shaped functions from~\cite{ks}.

The term \emph{bell-shaped sequence} apparently has not appeared in literature in the sense defined above prior to~\cite{kw}. Nevertheless, closely related concepts have been around for over a century, and have now become classical subjects. The introduction to~\cite{bgkp} provides an up-to-date comprehensive list of applications of total positivity and Pólya frequency sequences in various areas of mathematics, with references; here we also refer to older monographs~\cite{gm,karlin,pinkus}. Closely related problems about polynomials and entire functions with real zeroes go back to the works of Laguerre, Pólya and Schur. Completely monotone sequences originate in the solution of Hausdorff's moment problem.

We conclude this part with the following observation. Pólya frequency sequences form an important subclass of log-concave sequences. On the other hand, completely monotone sequences are log-convex. Thus, the class of bell-shaped sequences spans between the classes of log-concave and log-convex sequences, and provides an intermediate notion that is tailored for applications in probability (see the examples discussed later in this section) and possibly also in other areas of mathematics.

\subsection{Main results}

Before we state our main theorem, we need auxiliary definitions. 

Following~\cite{kw}, we say that a real function $\ph$ is \emph{stepwise increasing}, if it is integer-valued and nondecreasing. A real function $\ph$ is called \textit{increasing-after-rounding} if there exists a stepwise increasing function $\tilde{\ph}$ such that $\tilde{\ph} \le \ph \le \tilde{\ph} + 1$. \emph{Stepwise decreasing} and \emph{decreasing-after-rounding} functions are defined in an analogous way. It is straightforward to see that $\ph$ is increasing-after-rounding if $\lfloor \ph \rfloor$ or $\lceil \ph \rceil$ are stepwise increasing: in this case we may set $\tilde{\ph} = \lfloor \ph \rfloor$ or $\tilde{\ph} = \lceil \ph \rceil - 1$. The former condition is, however, slightly more general. We remark that in~\cite{kwasnicki,ks}, the term \emph{level crossing condition} is used to describe increasing-after-rounding functions.

A (summable) two-sided \emph{Pólya frequency sequence} is, up to multiplication by a constant, the probability mass function of the sum or difference of at most countably many independent Poissonian, geometric and Bernoulli random variables. A two-sided sequence $a(k)$ is said to be \emph{absolutely monotone-then-completely monotone} if the one-sided sequences $a(k)$ and $a(-k)$ are completely monotone. We refer to Sections~\ref{sec:amcm} and~\ref{sec:pff} for a detailed discussion

\begin{figure}
\includegraphics[width=0.8\textwidth]{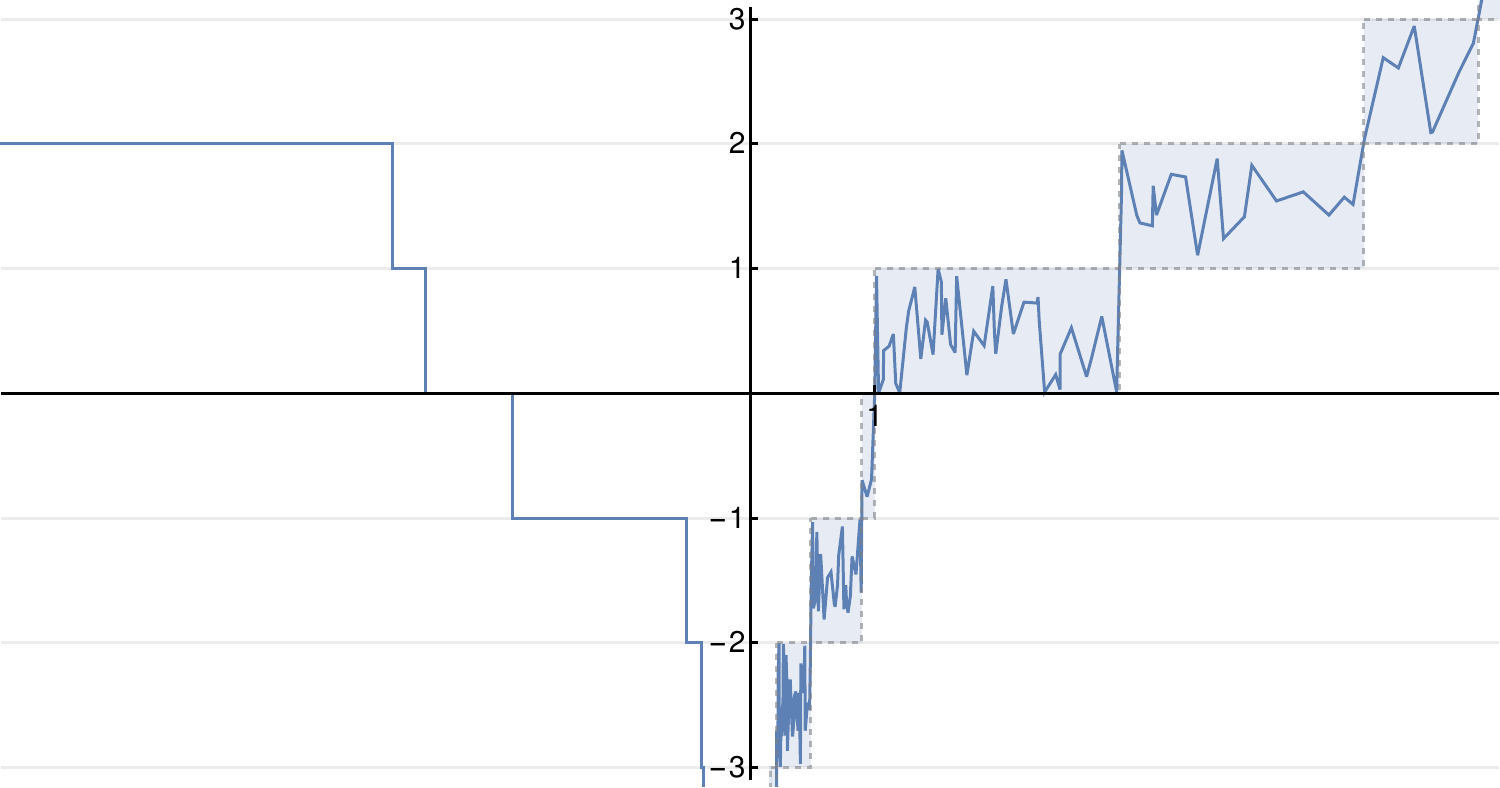}
\caption{A sample function $\ph$ in Theorem~\ref{thm:main}\ref{thm:main:c}.}
\label{fig:bs}
\end{figure}

\begin{theorem}
\label{thm:main}
Suppose that $a(k)$ is a two-sided sequence. The following conditions are equivalent:
\begin{enumerate}[label={\textup{(\alph*)}},leftmargin=2.5em]
\item
\label{thm:main:a}
$a(k)$ is a bell-shaped sequence;
\item
\label{thm:main:b}
$a(k)$ is the convolution of a summable Pólya frequency sequence $b(k)$ and an absolutely monotone-then-completely monotone sequence $c(k)$ which converges to zero as $k \to \pm \infty$;
\item
\label{thm:main:c}
The generating function of $a(k)$,
\formula{
 F(z) & = \sum_{k = -\infty}^\infty a(k) z^k ,
}
converges on the the unit circle $|z| = 1$, except possibly $z = 1$, and $F$ extends to a holomorphic function in the upper complex half plane, given by
\formula[eq:main]{
 F(z) & = \exp\biggl(b^+ z + \frac{b^-}{z} + c + \int_{-\infty}^\infty \biggl(\frac{1}{s - z} - \frac{s}{s^2 + 1}\biggr) \ph(s) ds \biggr)
}
when $\im z > 0$. Here $b^+, b^- \ge 0$, $c \in \R$, and $\ph$ is a Borel function on $\R$ such that (see Figure~\ref{fig:bs}):
 \begin{enumerate}[label={\textup{(\roman*)}},leftmargin=2.5em]
 \item\label{thm:main:c1} $\ph$ is stepwise decreasing on $(-\infty, 0)$;
 \item\label{thm:main:c2} $\ph$ is increasing-after-rounding on $(0, \infty)$;
 \item\label{thm:main:c3} $\ph \le 0$ on $(0, 1)$ and $\ph \ge 0$ on $(1, \infty)$;
 \item\label{thm:main:c4} $\ph$ satisfies the integrability condition
 \formula{
  \int_{-\infty}^\infty \frac{|\ph(s)|}{s^2 + 1} \, ds < \infty ;
 }
 \item\label{thm:main:c5} the function $F$ defined by~\eqref{eq:main} satisfies
 \formula{
  \lim_{t \to 0} (e^{i t} - 1) F(e^{i t}) & = 0 .
 }
 \end{enumerate}
\end{enumerate}
Additionally, every quadruple $b^+, b^-, c, \ph$ satisfying the conditions listed in item~\ref{thm:main:c} corresponds to a unique bell-shaped sequence $a(k)$.
\end{theorem}

Theorem~\ref{thm:main} is proved in Section~\ref{sec:proof}. For similar representations of absolutely monotone-then-completely monotone sequences and Pólya frequency sequences, see Lemmas~\ref{lem:amcm} and~\ref{lem:pf}, respectively. Factorisation of a bell-shaped sequence into the convolution of one-sided bell-shaped sequences is discussed in Corollary~\ref{cor:wh}.

\begin{remark}
\label{rem:c5}
There seems to be no simple way to rewrite condition~\ref{thm:main:c5} in Theorem~\ref{thm:main}\ref{thm:main:c} directly in terms of the function $\ph$. Nevertheless, it is a very natural condition, which, roughly speaking, corresponds to the fact that $\sum_{k = -\infty}^\infty \Delta a(k) = 0$. More precisely, we have the following two sufficient conditions for condition~\ref{thm:main:c5}.
\begin{itemize}[leftmargin=2.5em]
\item If $F$ is the generating function of a summable two-sided sequence $a(k)$, then $F$ is continuous on the unit circle $|z| = 1$, and hence condition~\ref{thm:main:c5} is automatically satisfied.
\item If $a(k)$ converges eventually monotonically to zero as $k \to \infty$ and as $k \to -\infty$, then the generating function $F$ of $a(k)$ is well-defined (as a conditionally convergent series) when $|z| = 1$, $z \ne 1$, and $(z - 1) F(z)$ is the generating function of $\Delta a(k)$. Furthermore, in this case $\Delta a(k)$ is summable (it has eventually constant sign as $k \to \infty$ and as $k \to -\infty$) and it sums up to zero. Thus, $(z - 1) F(z)$ extends to a continuous function on the unit circle $|z| = 1$, and it takes value $0$ at $z = 1$. This is equivalent to say that condition~\ref{thm:main:c5} is satisfied.
\end{itemize}
\end{remark}

\begin{remark}
\label{rem:ex}
If the generating function $F(z)$ of a given sequence $a(k)$ is given by an explicit formula when $|z| = 1$, $z \ne 1$, and this formula defines a holomorphic function in the upper complex half-plane $\im z > 0$ (denoted again by $F(z)$), then it is usually not very difficult to verify whether the conditions of Theorem~\ref{thm:main}\ref{thm:main:c} are satisfied. Indeed: it suffices to check that $F$ is zero-free in the upper complex half-plane, and then study the continuous version of the argument of $F(z)$, which we denote by $\pi \Phi(z)$. One needs to show that $\Phi$ is the sum of two terms: the Poisson integral of an appropriate function $\ph$ (which is necessarily equal to the boundary values of $\Phi$), and $b^+ \im z + b^- \im z^{-1}$. Here by saying that $\ph$ is an appropriate function we mean that conditions~\ref{thm:main:c1} through~\ref{thm:main:c4} in Theorem~\ref{thm:main}\ref{thm:main:c} hold true.

Later in this section we apply the above procedure to two classes of sequences; see Examples~\ref{ex:rw} and~\ref{ex:rrw}.
\end{remark}

\begin{remark}
\label{rem:alt}
Observe that
\formula{
 \frac{1}{s - z} - \frac{s}{s^2 + 1} & = \frac{1}{s^2 + 1} \, \frac{1 + s z}{s - z} \, .
}
Let us agree that $(1 + s z) / (s - z) = z$ when $s = \infty$, so that $(1 + s z) / (s - z)$ becomes a continuous function of $s$ on $\R \cup \{\infty\}$, the one-point compactification of $\R$. Furthermore, let us denote
\formula{
 \sigma(ds) & = b^+ \delta_\infty(ds) - b^- \delta_0(ds) + \frac{\ph(s)}{s^2 + 1} \, ds .
}
Then formula~\eqref{eq:main} can be written as
\formula[eq:main:alt]{
 F(z) & = \exp\biggl(c + \int_{\R \cup \{\infty\}} \frac{1 + s z}{s - z} \, \sigma(ds)\biggr) .
}
Here $\sigma$ is a finite signed measure on $\R \cup \{\infty\}$, with an appropriate density function on $\R \setminus \{0\}$.
\end{remark}

\begin{remark}
\label{rem:symmetry}
Note that a sequence $a(k)$ is bell-shaped if and only if its mirror image $a(-k)$ is bell-shaped. If the generating function of $a(k)$ is given by the representation formula~\eqref{eq:main}, then the generating function of $a(-k)$ is equal to $F(1/z)$. Furthermore, $F(1/z)$ is again given by the right-hand side of~\eqref{eq:main}, with parameters $b^+, b^-, c, \ph(s)$ replaced by $b^-, b^+, c, -\ph(1/s)$.

One-sided bell-shaped sequences correspond to $b^- = 0$ and $\ph(s) \ge 0$ in Theorem~\ref{thm:main}; see Theorem~1.1 in~\cite{kw}. Combining this with the above observation, we find that bell-shaped sequences $a(k)$ such that $a(k) = 0$ for $k > 0$ correspond to $b^+ = 0$ and $\ph(s) \le 0$ in Theorem~\ref{thm:main}.
\end{remark}

Our main result, Theorem~\ref{thm:main}, clearly resembles the corresponding statement for one-sided bell-shaped sequences (Theorem~1.1 in~\cite{kw}), as well as the analogous result for bell-shaped functions (Theorems~1.1 and~1.3 in~\cite{ks}). Proofs of all these theorems have a similar structure, developed mostly in~\cite{kwasnicki,ks}. We stress, however, that the proof of Theorem~\ref{thm:main} given below contains two essentially new elements, already mentioned in the introduction. First, we need a replacement for Post's inversion formula for the Laplace transform: its discrete variant employed in~\cite{kw} does not apply to two-sided sequences. For this reason we prove in Proposition~\ref{prop:post} an inversion formula that involves iterated differences of the Fourier--Laplace transform of the generating function. Next, unlike in the case of one-sided sequences, holomorphic functions defined by the right-hand side of~\eqref{eq:main} are no longer exponentials of Pick functions; instead, they are exponentials of certain differences of Pick functions. The pointwise limit of Pick functions is necessarily a Pick function, but this is no longer true for differences of Pick functions. We were unable to find a compactness result that would suit our needs in available literature. The property that is needed in the proof of Theorem~\ref{thm:main} is given in Lemma~\ref{lem:compact}. Its proof turned out to be surprisingly long and technical.

Theorem~\ref{thm:main} shows that, the convolution of bell-shaped sequences, whenever defined, corresponds to addition of the parameters $b^+, b^-, c, \ph$ in the representation formula~\eqref{eq:main} of the corresponding generating functions. Using this fact it is easy to construct bell-shaped sequences whose convolution is not bell-shaped (because the sum of corresponding functions $\ph$ is not increasing-after-rounding on $(0, \infty)$). We have, however, the following interesting result.

\begin{corollary}[Wiener--Hopf factorisation of bell-shaped sequences]
\label{cor:wh}
Suppose that $b(k)$ is a one-sided bell-shaped sequences, and $c(k)$ is the mirror image of a one-sided bell-shaped sequence (that is, $c(k)$ is bell-shaped and $c(k) = 0$ for $k > 0$). Suppose furthermore that the convolution $a(k) = b * c(k)$ is well-defined, and that $a(k)$ converges to zero as $k \to \pm \infty$. Then $a(k)$ is bell-shaped. Conversely, every bell-shaped sequence can be factorised in the way described above.
\end{corollary}

\begin{proof}
The direct part is a simple consequence of Theorem~\ref{thm:main} and its one-sided version given in~\cite{kw}. Indeed: suppose that $\ph_1$ and $\ph_2$ correspond to the functions $\ph$ in the representation~\eqref{eq:main} of the generating functions $G$ and $H$ of $b(k)$ and $c(k)$, respectively. The generating function of $a(k)$ is easily shown to be equal to $F(z) = G(z) H(z)$, and therefore $F$ has representation~\eqref{eq:main} with $\ph = \ph_1 + \ph_2$. However, by Remark~\ref{rem:symmetry}, $\ph_1(s) = 0$ for $s \in [0, 1]$ and $\ph_2(s) = 0$ for $s \in [1, \infty)$. It is thus easy to see that $\ph$ has properties~\ref{thm:main:c1} through~\ref{thm:main:c4}. Furthermore, since $a(k)$ converges eventually monotonically to zero as $k \to \infty$ and as $k \to -\infty$, condition~\ref{thm:main:c5} is satisfied by Remark~\ref{rem:c5}. Consequently, $F$ is the generating function of a bell-shaped sequence $a(k) = b * c(k)$.

The converse part of the corollary has an even simpler proof: it suffices to define $\ph_1(s) = \max\{\ph(s), 0\}$ and $\ph_2(s) = \min\{\ph(s), 0\}$, and observe that $\ph(s) = \ph_1(s) + \ph_2(s)$, the parameters $b^+, 0, c, \ph_1$ correspond to a one-sided bell-shaped sequence $b(k)$ (by Remark~\ref{rem:symmetry}), while the parameters $0, b^-, 0, \ph_2$ correspond to a sequence $c(k)$ which is the mirror image of a one-sided bell-shaped sequence (again by Remark~\ref{rem:symmetry}).
\end{proof}

By a very similar argument, we are able to describe when convolution powers or convolution roots of bell-shaped sequences are bell-shaped. We leave the details of the proof to the interested reader, and we refer to Corollaries~1.4 and~1.5 in~\cite{ks} for analogous results for bell-shaped functions.

\begin{corollary}[infinitely divisible bell-shaped sequences]
Suppose that a probability mass function $a(k)$ is a two-sided bell-shaped sequence. Then $a(k)$ is the probability mass function of an infinitely divisible distribution if and only if the corresponding function $\ph$ defined in Theorem~\ref{thm:main} is equal to zero on $(-\infty, 0)$. In this case the convolution roots of $a(k)$ are bell-shaped, too.
\end{corollary}

\begin{corollary}[convolution powers of bell-shaped sequences]
Suppose that $a(k)$ is a two-sided bell-shaped sequence, and the convolution powers of $a(k)$ are well-defined. Then the following two conditions are equivalent:
\begin{enumerate}[label={\textup{(\alph*)}},leftmargin=2.5em]
\item every convolution power of $a(k)$ is bell-shaped;
\item the corresponding function $\ph$ defined in Theorem~\ref{thm:main} is nondecreasing on $(0, \infty)$.
\end{enumerate}
\end{corollary}

\begin{remark}
\label{rem:egnbc}
Bell-shaped sequences are related to \emph{generalised negative binomial convolutions}. The probability mass function of a negative binomial distribution is given by~\eqref{eq:nb} below, and generalised negative binomial convolutions are probability distributions that arise as weak limits of finite convolutions of negative binomial distributions. For a detailed discussion, we refer the reader to Chapter~8 of~\cite{bondesson:book}.

It is easy to see that a one-sided sequence $a(k)$ is the probability mass function of a generalised negative binomial convolution if and only if it has the representation~\eqref{eq:main} given in Theorem~\ref{thm:main}, with $b^+ \ge 0$, $b^- = 0$, appropriate $c \in \R$, $\ph(s) = 0$ for $s \in (-\infty, 1]$, and $\ph$ increasing on $[1, \infty)$. This follows by a routine calculation from equation~(8.1.2) in~\cite{bondesson:book}, with $\ph(s) = V([1/s, 1))$ for $s > 1$; we omit the details. Thus, probability mass functions of generalised negative binomial convolutions are bell-shaped.

If we assume that $\ph(s) = 0$ for $s \in (-\infty, 0)$ and $\ph$ is increasing on $(0, \infty)$ in Theorem~\ref{thm:main}, then the corresponding sequence $a(k)$ is the convolution of a generalised negative binomial convolution with the mirror image of a generalised negative binomial convolution. In analogy to generalised gamma convolutions, this class of sequences may be called \emph{extended generalised negative binomial convolutions}. Of course, this is again a subclass of bell-shaped sequences.

An analogous result for extended generalised gamma convolutions and bell-shaped functions is given in Corollary~1.5 in~\cite{ks}.
\end{remark}

\subsection{Examples}

The following one-sided bell-shaped sequences (that is, sequences $a(k)$ equal to $0$ for $k < 0$) have already been discussed in~\cite{kw}:
\begin{itemize}[leftmargin=2.5em]
\item the delta sequence:
\formula{
 a(k) & = \begin{cases} 1 & \text{for $k = 0$;} \\ 0 & \text{otherwise} \end{cases}
}
(this corresponds to $\ph(s) = 0$ and $b^+ = b^- = 0$);
\item geometric sequences:
\formula{
 a(k) & = q^k ,
}
where $q \in (0, 1)$ (here $\ph(s) = \ind_{(1 / q, \infty)}(s)$ and $b^+ = b^- = 0$);
\item probability mass functions of Poisson random variables:
\formula{
 a(k) & = e^{-\lambda} \, \frac{\lambda^k}{k!} \, ,
}
where $\lambda > 0$ (here $\ph(s) = 0$, $b^+ = \lambda$ and $b^- = 0$);
\item probability mass functions of Bernoulli and, more generally, binomial distributions:
\formula{
 a(k) & = \binom{n}{k} p^k (1 - p)^{n - k} ,
}
where $n = 1, 2, \ldots$ and $p \in (0, 1)$ (here $\ph(s) = n \ind_{(-\infty, 1 - 1/p)}(s)$ and $b^+ = b^- = 0$);
\item one-sided summable Pólya frequency sequences; that is, convolutions of a finite or countable infinite number of the above sequences (here $b^+, b^- \ge 0$, $\ph$ is stepwise decreasing on $(-\infty, 0]$, zero on $[0, 1]$, and stepwise increasing on $[1, \infty)$; cf.\@ Lemma~\ref{lem:pf});
\item completely monotone sequences (here $b^+ = b^- = 0$, $\ph(s) = 0$ on $(-\infty, 1]$ and $\ph(s) \in [0, 1]$ on $[1, \infty)$; cf.\@ Lemma~\ref{lem:amcm});
\item probability mass functions of negative binomial distributions:
\formula[eq:nb]{
 a(k) & = \binom{\lambda + k - 1}{k} p^\lambda (1 - p)^k ,
}
where $\lambda \in (0, \infty)$ and $p \in (0, 1)$ (here $\ph(s) = \lambda \ind_{(1 / (1 - p), \infty)}(s)$ and $b^+ = b^- = 0$);
\item probability mass functions of discrete stable distributions, with generating function
\formula{
 \sum_{k = 0}^\infty a(k) z^k & = \exp(-\lambda (1 - z)^\nu) ,
}
where $\lambda \in (0, \infty)$ and $\nu \in (0, 1]$ (here $\ph(s) = \tfrac{\lambda}{\pi} \sin(\nu \pi) (s - 1)^\nu \ind_{(1, \infty)}(s)$ and $b^+ = b^- = 0$ if $\nu \in (0, 1)$; for $\nu = 1$, the discrete stable distribution is simply the Poisson distribution), see~\cite{sv} and Section~6 in~\cite{kw}.
\end{itemize}
In view of Remark~\ref{rem:egnbc}, the probability mass functions of generalised Waring distributions (Example~8.2.3 in~\cite{bondesson:book}) and logarithmic series distributions (Example~8.2.4 in~\cite{bondesson:book}) are further examples of one-sided bell-shaped functions.

The class of one-sided bell-shaped sequences is not closed under convolutions. However, the main result of~\cite{kw} describes one-sided bell-shaped sequences as convolutions of completely monotone sequences which converge to zero and one-sided summable Pólya frequency sequences. This is in perfect analogy with our Theorem~\ref{thm:main}.

By Corollary~\ref{cor:wh}, any convolution of a one-sided bell-shaped sequence with the mirror image of a one-sided bell-shaped sequence, whenever well-defined and convergent to zero at $\pm\infty$, is again bell-shaped. Thus, for example, the following two-sided sequences are bell-shaped:
\begin{itemize}
\item two-sided geometric sequences, which are convolutions of one-sided geometric sequences:
\formula{
 a(k) & = \begin{cases} q_+^k & \text{if $k \ge 0$,} \\ q_-^{-k} & \text{if $k < 0$,} \end{cases}
}
where $q_+, q_- \in (0, 1)$ (here $\ph(s) = \ind_{(1/q_+, \infty)}(s) - \ind_{(0, q_-)}(s)$ and $b^+ = b^- = 0$);
\item probability mass functions of Skellam's distributions, that is, convolutions of one-sided Poisson distributions:
\formula{
 a(k) & = e^{-\lambda_+ - \lambda_-} \lambda_+^{k / 2} \lambda_-^{-k / 2} I_k(2 \sqrt{\lambda_+ \lambda_-}) ,
}
where $\lambda_+, \lambda_- \in (0, \infty)$ and $I_k$ is the modified Bessel function of the first kind (here $\ph(s) = 0$ and $b^+ = \lambda_+$, $b^- = \lambda_-$);
\item probability mass functions of two-sided discrete stable distributions, that is, convolutions of one-sided discrete stable distributions corresponding to the same parameter $\nu \in (0, 1]$ (here $\ph(s) = \tfrac{1}{\pi} \sin(\nu \pi) (\lambda_+ (s - 1)^\nu \ind_{(1, \infty)}(s) - \lambda_- (s^{-1} - 1)^\nu \ind_{(0, 1)}(s))$ and $b^+ = b^- = 0$ if $\nu \in (0, 1)$; for $\nu = 1$, this coincides with Skellam's distribution).
\end{itemize}

It is perhaps surprising that bell-shaped functions give rise to bell-shaped sequences if sampled at equal time intervals. More precisely, we have the following result.

\begin{theorem}
\label{thm:sample}
If $f$ is a bell-shaped function, then the sequence $(f(k) : k \in \Z)$ is bell-shaped.
\end{theorem}

Theorem~\ref{thm:sample} is proved in Section~\ref{sec:sample}.

Clearly, not every bell-shaped sequence arises in this way. For example, probability mass functions of binomial distributions are bell-shaped and have finitely many nonzero terms. On the other hand, Hirschman proved in~\cite{hirschman} that there are no compactly supported bell-shaped functions, so if $f$ is a bell-shaped function, then the sequence $f(k)$ has infinitely many nonzero terms.

The anonymous referee asked natural and interesting questions regarding the characterisation of bell-shaped sequences that arise by sampling bell-shaped functions, as in Theorem~\ref{thm:sample}, and about additional assumptions on $f$ under which the converse of Theorem~\ref{thm:sample} holds. Analogous results for completely monotone sequences and functions are given in~\cite{aj}, but the case of bell-shaped sequences and functions appears to be more subtle. For instance, probability mass functions of Poisson distributions are not of the form $f(k)$ for a one-sided bell-shaped function $f$, since one can prove that $f$ decays at infinity at most exponentially fast. In particular, the function $(\Gamma(x + 1))^{-1} \lambda^x \ind_{(-1, \infty)}(x)$ is not bell-shaped.

Theorem~\ref{thm:sample} leads to another series of examples of one-sided and two-sided bell-shaped sequences:
\begin{itemize}[leftmargin=2.5em]
\item discretised density functions of the normal distribution,
\formula{
 a(k) & = \exp(-\alpha k^2 - \beta k - \gamma) ,
}
where $\alpha > 0$ and $\beta, \gamma \in \R$;
\item discretised density functions of (one-sided) inverse Gaussian distributions
\formula{
 a(k) & = \begin{cases} e^{-\alpha / k} k^{-p} & \text{if $k > 0$,} \\ 0 & \text{otherwise,} \end{cases}
}
where $p > 0$ and $\alpha > 0$;
\item negative powers of quadratic expressions,
\formula{
 a(k) & = \frac{1}{(\alpha k^2 + \beta k + \gamma)^p} \, ,
}
where $p > 0$, $\alpha > 0$, $\beta, \gamma \in \R$, and $\beta^2 > 4 \alpha \gamma$;
\item sequences of the form
\formula{
 a(k) & = \frac{1}{((k - \alpha)^2 + \beta^2) ((k - \alpha)^2 + \gamma^2)} \, ,
}
where $\alpha \in \R$ and $\beta, \gamma > 0$.
\end{itemize}
We refer to Sections~6.4 and~6.5 in~\cite{kwasnicki} for a discussion of the corresponding bell-shaped functions. Interestingly, there seems to be no other simple proof that these sequences are indeed bell-shaped. A direct verification of the definition seems hopeless, and the characterisation provided in Theorem~\ref{thm:main} may be difficult to apply. In particular, there seems to be no simple formula for the functions $\ph$ (and for the generating functions $F$) corresponding to the above examples.

The following beautiful result for one-sided sequences is due to Bondesson~\cite{bondesson}. It is a discrete analogue of the similar theorem about hitting times of one-dimensional diffusions, given in~\cite{js}.

\begin{example}[hitting time of random walks]
Consider a \emph{random walk} $X_n$ on $\N = \{0, 1, 2, \ldots\}$: a discrete-time Markov chain with steps $\pm 1$ as long as $X_n > 0$, and with an absorbing state $0$. Suppose that $X_0 = x$, and let $N = \min\{n \in \N : X_n = 0\}$ denote the time to absorption. Then the probability mass function of $\tfrac{1}{2} (N - x)$ is a one-sided bell-shaped sequence. This follows immediately from Theorem~1 in~\cite{bondesson} and Theorem~1.1 in~\cite{kw}; see also Theorem~3 therein.
\end{example}

Motivated by Bondesson's result, below we give two examples of two-sided bell-shaped sequences, which illustrate a concept that is significantly extended in the follow-up paper~\cite{wszola} by the second named author.

\begin{figure}
\includegraphics[width=0.8\textwidth]{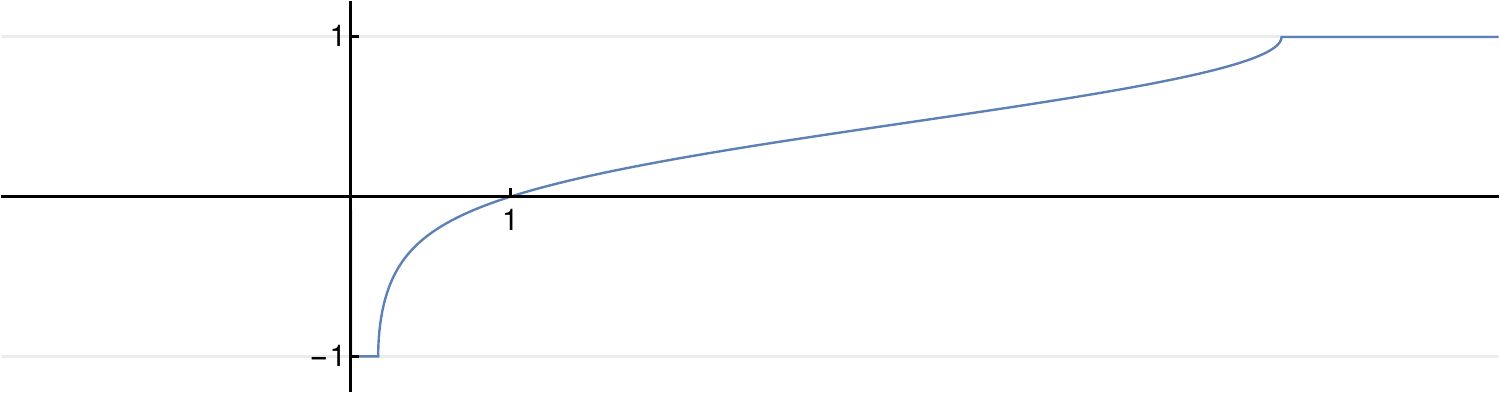} \\
(a) \\[0.5em]
\includegraphics[width=0.8\textwidth]{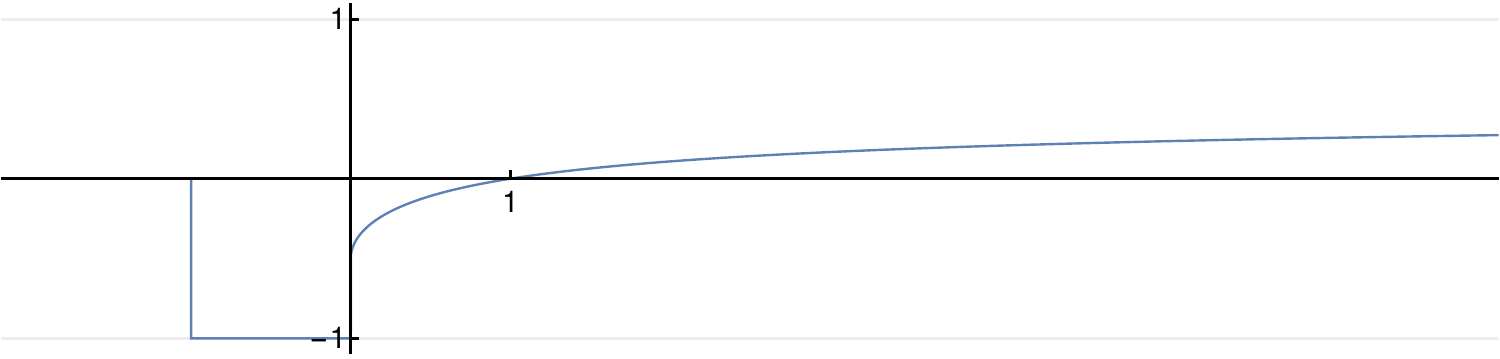} \\
(b)
\caption{The function $\ph$ in: (a)~Example~\ref{ex:rw}; (b)~Example~\ref{ex:rrw}.}
\label{fig:rw}
\end{figure}

\begin{example}[discrete Poisson kernel of a half-plane for a simple random walk]
\label{ex:rw}
Consider a simple random walk $(X_n, Y_n)$ on the square lattice $\Z^2$, performing steps $(\pm 1, 0)$ or $(0, \pm 1)$ with probabilities $\tfrac{1}{4}$. Suppose that $(X_0, Y_0) = (0, y)$, and let $N = \min\{n \in \N : Y_n = 0\}$ denote the hitting time of $\Z \times \{0\}$. Then $X_N$ has a bell-shaped probability mass function.

In order to prove the above claim, we follow the approach described in Remark~\ref{rem:ex}. By definition, the generating function $F_y$ of the probability mass function of $X_N$ satisfies
\formula{
 F_y(z) & = \frac{z F_y(z) + z^{-1} F_y(z) + F_{y + 1}(z) + F_{y - 1}(z)}{4}
}
when $|z| = 1$ and $y \ge 1$. Additionally, we have $F_0(z) = 1$ and $|F_y(z)| \le 1$ when $|z| = 1$ and $y \ge 1$. The solution of this second-order linear recurrence equation is given by
\formula{
 F_y(z) & = \biggl(2 - \frac{z + z^{-1}}{2} - \sqrt{1 - \frac{z + z^{-1}}{2}} \sqrt{3 - \frac{z + z^{-1}}{2}}\biggr)^y
}
when $|z| = 1$; we omit the details. Observe that for every $y \ge 0$, the right-hand side extends to a holomorphic function in the upper complex half-plane $\im z > 0$. Indeed: when $\im z > 0$, then one easily finds that $1 - (z + z^{-1}) / 2 \in \C \setminus (-\infty, 0]$, and therefore $\sqrt{1 - (z + z^{-1}) / 2}$ and $\sqrt{3 - (z + z^{-1}) / 2}$ are well-defined and holomorphic in this region. In Appendix~\ref{app:b} we show that $F_1$ is given by~\eqref{eq:main}, with $b^+ = b^- = 0$, with an appropriate constant $c$, and with
\formula{
 \ph(s) & = \lim_{t \to 0^+} \frac{1}{\pi} \, \Arg F_1(s + i t) .
}
Additionally, we verify that (see Figure~\ref{fig:rw}(a))
\formula{
 \ph(s) & = \begin{cases}
  0 \vphantom{\dfrac{\sqrt{s^{-1}}}{s^{-1}}} & \text{if $s < 0$;} \\
  -1 \vphantom{\dfrac{\sqrt{s^{-1}}}{s^{-1}}} & \text{if $0 < s < 3 - 2 \sqrt{2}$;} \\
  -\dfrac{1}{\pi} \arccot \dfrac{4 - s - s^{-1}}{\sqrt{s + s^{-1} - 2} \sqrt{6 - s - s^{-1}}} & \text{if $3 - 2 \sqrt{2} < s < 1$;} \\
  \dfrac{1}{\pi} \arccot \dfrac{4 - s - s^{-1}}{\sqrt{s + s^{-1} - 2} \sqrt{6 - s - s^{-1}}} & \text{if $1 < s < 3 - \sqrt{2}$;} \\
  1 \vphantom{\dfrac{\sqrt{s^{-1}}}{s^{-1}}} & \text{if $s > 3 + 2 \sqrt{2}$.}
 \end{cases}
}
Since $F_y(z) = (F_1(z))^y$, we conclude that for every $y = 0, 1, 2, \ldots$ the function $F_y$ has the representation~\eqref{eq:main} with $b^+ = b^- = 0$, with constant $c$ replaced by $c y$, and with $\ph$ replaced by $y \ph$. Since $\ph$ is nondecreasing on $(0, \infty)$ and constant on $(-\infty, 0)$, it follows by Theorem~\ref{thm:main} that $F_y$ is the generating function of a two-sided bell-shaped sequence, that is, $X_N$ has a bell-shaped probability mass function.
\end{example}

\begin{example}[discrete Poisson kernel of a half-plane for two independent simple random walks]
\label{ex:rrw}
Consider a pair $(X_n, Y_n)$ of independent simple random walks on $\Z$, so that $(X_n, Y_n)$ is a random walk in $\Z^2$ which makes steps $(\pm 1, \pm 1)$ with probabilities $\tfrac{1}{4}$. In other words, $(X_n, Y_n)$ is the simple random walk on the diagonal square lattice (the usual square lattice rotated by $\tfrac{\pi}{4}$). Suppose that $(X_0, Y_0) = (0, y)$, and let $N = \min\{n \in \N : Y_n = 0\}$ denote the hitting time of $\Z \times \{0\}$. Then $\tfrac{1}{2} (X_N + y)$ has a bell-shaped probability mass function.

As in the previous example, we follow the approach described in Remark~\ref{rem:ex}. The generating function $F_y$ of the probability mass function of $\tfrac{1}{2} (X_N + y)$ satisfies
\formula{
 F_y(z) & = \frac{z F_{y + 1}(z) + F_{y + 1}(z) + F_{y - 1}(z) + z^{-1} F_{y - 1}(z)}{4}
}
when $|z| = 1$ and $y \ge 1$. Furthermore, $F_0(z) = 1$ and $|F_y(z)| \le 1$ when $|z| = 1$ and $y \ge 1$. The solution of this second-order linear recurrence equation is given by
\formula{
 F_y(z) & = \biggl(\frac{2 + i (z^{1/2} - z^{-1/2})}{z + 1}\biggr)^y
}
when $|z| = 1$ and $\im z > 0$; again we omit the details. Clearly, the right-hand side defines a holomorphic function in the upper complex half-plane $\im z > 0$. In Appendix~\ref{app:b} we prove that $F_1$ is given by~\eqref{eq:main}, with $b^+ = b^- = 0$, with an appropriate constant $c$, and with
\formula{
 \ph(s) & = \lim_{t \to 0^+} \frac{1}{\pi} \, \Arg F_1(s + i t) .
}
Additionally, we verify that (see Figure~\ref{fig:rw}(b))
\formula{
 \ph(s) & = \begin{cases}
  0 \vphantom{\dfrac{\sqrt{s^{-1}}}{s^{-1}}} & \text{if $s < -1$;} \\
  -1 \vphantom{\dfrac{\sqrt{s^{-1}}}{s^{-1}}} & \text{if $-1 < s < 0$;} \\
  \dfrac{1}{\pi} \arctan \dfrac{s^{1/2} - s^{-1/2}}{2} & \text{if $s > 0$.}
 \end{cases}
}
Since $F_y(z) = (F_1(z))^y$, as in the previous example we conclude that for every $y = 0, 1, 2, \ldots$ the function $F_y$ has the representation~\eqref{eq:main} with $b^+ = b^- = 0$, with constant $c$ replaced by $c y$, and with $\ph$ replaced by $y \ph$. Additionally, $\ph$ is nondecreasing on $(0, \infty)$ and stepwise decreasing on $(-\infty, 0)$, and hence, by Theorem~\ref{thm:main}, $F_y$ is the generating function of a two-sided bell-shaped sequence, as desired.
\end{example}

\subsection{Organisation of the paper}

The remaining part of this article is divided into two sections. In Preliminaries we recall the notions of Pick functions (Section~\ref{sec:pick}), generating functions (Section~\ref{sec:gen}), absolute and complete monotonicity (Section~\ref{sec:amcm}), and Pólya frequency sequences (Section~\ref{sec:pff}). We also prove various auxiliary results, including an inversion formula (Section~\ref{sec:post}) and a compactness lemma (Section~\ref{sec:compact}). In Section~\ref{sec:proof} of this paper we prove our main result, Theorem~\ref{thm:main}. We first discuss the easy parts (Section~\ref{sec:easy}), before we prove the most difficult step, implication \ref{thm:main:a}~$\Rightarrow$~\ref{thm:main:c} (Section~\ref{sec:hard}). Detailed calculations for the proof of the inversion formula and for the examples discussed above are given in two appendices.

%
%

\section{Preliminaries}
\label{sec:pre}

Throughout the paper, we denote by $x\rf{n}$ the \emph{rising factorial}:
\formula{
 x\rf{n} & = x (x + 1) \ldots (x + n - 1) .
}
By $\log z$ for $z \in \C \setminus (-\infty, 0]$ we denote the principal branch of the complex logarithm. We also set $\log(-s) = \log s + i \pi$, so that $\log$ becomes continuous on the closed upper complex half-plane with $0$ removed. By $\Arg z = \im \log z$ we denote the principal argument of $z \in \C \setminus \{0\}$.


\subsection{Pick functions}
\label{sec:pick}

A \emph{Pick function}, also known under the names \emph{Herglotz function}, \emph{Nevanlinna function} or \emph{Nevanlinna--Pick} function, is a holomorphic function $f$ which maps the upper complex half-plane $\{z \in \C : \im z > 0\}$ to its closure; that is, $\im f(z) \ge 0$ whenever $\im z > 0$. By Theorem~II.I in~\cite{donoghue}, every Pick function admits the \emph{Stieltjes representation}
\formula[eq:pick:int]{
 f(z) & = b z + c + \int_{\R} \biggl(\frac{1}{s - z} - \frac{s}{s^2 + 1}\biggr) \mu(ds) ,
}
where $b \ge 0$, $c \in \R$ and $\mu$ is a measure on $\R$ satisfying the integrability condition $\int_{\R} (s^2 + 1)^{-1} \mu(ds) < \infty$. The parameters $b, c, \mu$ are otherwise arbitrary, and they are uniquely determined by
\formula{
 b & = \lim_{t \to \infty} \frac{\im f(i t)}{t} \, , & c & = \re f(i) ,
}
and, in the sense of vague convergence of measures,
\formula[eq:pick:mu]{
 \mu(ds) & = \lim_{t \to 0^+} \frac{1}{\pi} \, \im f(s + i t) ds ;
}
see Lemma~II.1 in~\cite{donoghue} or equation~(3.10) in~\cite{remling}. The measure $\mu$ is the \emph{Stieltjes measure} of $f$.

As in Remark~\ref{rem:alt}, formula~\eqref{eq:pick:int} can be rewritten as
\formula[eq:pick:alt]{
 f(z) & = c + \int_{\R \cup \{\infty\}} \frac{1 + s z}{s - z} \, \sigma(ds) ,
}
where $c \in \R$ and $\sigma$ is a finite measure on $\R \cup \{\infty\}$. More precisely, $\sigma(ds) = b \delta^\infty(ds) + (s^2 + 1)^{-1} \mu(ds)$, and in the above equation we agree that the integrand is equal to $z$ when $s = \infty$.

It is easy to see that if $f$ is a Pick function and $f$ is not constant $0$, then $\log f$ is another Pick function. Observe that $0 \le \im \log f(z) \le \pi$ when $\im z > 0$. By formula~\eqref{eq:pick:mu} (applied to $\log f$), the Stieltjes measure corresponding to $\log f$ necessarily has a density function with values in $[0, 1]$. Additionally, the corresponding coefficient $b$ is equal to zero. This brings us to the \emph{exponential representation} of the Pick function $f$:
\formula[eq:pick:exp]{
 f(z) & = \exp\biggl(c + \int_{-\infty}^\infty \biggl(\frac{1}{s - z} - \frac{s}{s^2 + 1}\biggr) \ph(s) ds\biggr) ,
}
where $c \in \R$ and $\ph$ is a Borel function on $\R$ with values in $[0, 1]$. Furthermore, $c = \log |f(i)|$, $\ph$ is determined uniquely up to equality almost everywhere, and
\formula[eq:pick:phi]{
 \ph(s) & = \lim_{t \to 0^+} \frac{1}{\pi} \, \Arg f(s + i t)
}
for almost every $s \in \R$. For further discussion, we refer to equation II.6 in~\cite{donoghue} and to Section~7.2 in~\cite{remling}.


\subsection{Sequences and generating functions}
\label{sec:gen}

A (real) \emph{two-sided sequence} is a function $a : \Z \to \R$. For clarity, we often write $a(k)$ rather than $a_k$ for the entries of the sequence, and we usually write `the sequence $a(k)$' instead of the more formal `the sequence $a$'. We identify a \emph{one-sided sequence} with the corresponding two-sided sequence $a(k)$ such that $a(k) = 0$ when $k < 0$.

The \emph{forward difference operator} is defined by
\formula{
 \Delta a(k) & = a(k + 1) - a(k) ,
}
and its powers $\Delta^n$ are defined in a straightforward way for $n = 0, 1, 2, \ldots$\, Clearly, if $a(k)$ is a summable sequence, then $\Delta a(k)$ summable. Furthermore, if $a(k)$ converges eventually monotonically to zero as $k \to \infty$ and as $k \to -\infty$, then $\Delta a(k)$ has an eventually constant sign as $k \to \infty$ and as $k \to -\infty$, and hence $\Delta a(k)$ is a summable sequence, which sums up to zero.

For a two-sided sequence $a(k)$ we define its generating function as
\formula{
 F(z) & = \sum_{k = -\infty}^\infty a(k) z^k
}
whenever the series converges (absolutely or conditionally). Clearly, the generating function of $\Delta a(k)$ is given by $(z - 1) F(z)$ whenever $F(z)$ is defined.

If $a(k)$ is a bounded one-sided sequence, the generating function $F$ is well defined and holomorphic in the unit disc $|z| < 1$ in the complex plane. If $a(k)$ is additionally summable, then $F(z)$ converges absolutely and it is continuous in the closed unit disc $|z| \le 1$. If a one-sided sequence $a(k)$ converges eventually monotonically to zero as $k \to \infty$, then, by Dirichlet's test, the generating function $F(z)$ of $a(k)$ converges in the closed disc $|z| \le 1$ except possibly $z = 1$. Furthermore, in this case the sequence of increments $\Delta a(k)$ is summable, its generating function is a continuous function on the closed unit disc $|z| \le 1$, and for $z \ne 1$ it is equal to $(z - 1) F(z)$. It follows that $F$ is continuous in the closed disc $|z| \le 1$, except possibly $z = 1$.

For a two-sided summable sequence $a(k)$, the generating function $F(z)$ converges absolutely on the unit circle $|z| = 1$. For general $a(k)$, the generating function $F(z)$ may diverge at every point $z$. However, if $a(k)$ converges eventually monotonically to zero as $k \to \infty$ and as $k \to -\infty$, then, by the discussion in the preceding paragraph, $F(z)$ converges on the unit circle $|z| = 1$ except possibly $z = 1$, and it is a continuous function on this set. Furthermore, in this case $F$ determines uniquely the sequence $a(k)$. Indeed: the continuous function $(z - 1) F(z)$ on the unit circle is the generating function of the summable sequence $\Delta a(k)$, which is thus determined uniquely by $F$: it is the sequence of Fourier coefficients of $(z - 1) F(z)$. Furthermore, since $a(k)$ converges to $0$ as $k \to -\infty$, the sequence $\Delta a(k)$ uniquely determines the sequence $a(k)$.

The \emph{convolution} of two-sided sequences $a(k)$ and $b(k)$ is defined in the usual way:
\formula{
 (a * b)(k) & = \sum_{j = -\infty}^\infty a(j) b(k - j)
}
whenever the series converges (absolutely or conditionally) for every $k \in \Z$. Clearly, $\Delta(a * b) = \Delta a * b = a * \Delta b$ whenever $a * b$ is well-defined. Suppose that $F$ and $G$ are generating functions of two-sided sequences $a(k)$ and $b(k)$. If $a(k)$ and $b(k)$ are summable sequences, then the generating function of their convolution is equal to $F(z) G(z)$ when $|z| = 1$. More generally, if $a(k)$ is summable and $b(k)$ converges eventually monotonically to zero as $k \to \infty$ and as $k \to -\infty$, then the convolution of these sequences converges absolutely, $(a * b)(k)$ converges to zero as $k \to \pm \infty$, and the generating function of the convolution is again given by $F(z) G(z)$ when $|z| = 1$, $z \ne 1$.

We recall the Abel's \emph{summation by parts} formula:
\formula{
 \sum_{k = -\infty}^\infty a(k) \Delta b(k) & = -\sum_{k = -\infty}^\infty \Delta a(k - 1) b(k)
}
whenever either of the sums converges and $a(k) b(k)$ converges to zero as $k \to \pm\infty$. Iterating this identity $n$ times, we find that
\formula{
 \sum_{k = -\infty}^\infty a(k) \Delta^n b(k) & = (-1)^n \sum_{k = -\infty}^\infty \Delta^n a(k - n) b(k) ,
}
provided that either of the sums converges and for every $j = 0, 1, \ldots, n - 1$ the sequence $\Delta^j a(k - j) \Delta^{n - 1 - j} b(k)$ converges to zero as $k \to \pm\infty$.

A two-sided sequence $a(k)$ is said to change sign at least $n$ times if there is a subsequence $a(k_0), a(k_1), \ldots, a(k_n)$ of alternating signs, in the sense that $a(k_{j - 1}) a(k_j) < 0$ for $j = 1, 2, \ldots, n$. If $a(k)$ changes sign at least $n$ times, but not $n + 1$ times, then we simply say that $a(k)$ \emph{changes sign $n$ times}.

By a discrete counterpart of Rolle's theorem, if $a(k)$ changes sign $n$ times and $a(k)$ converges to zero as $k \to \pm\infty$, then $\Delta a(k)$ changes sign at least $n + 1$ times: there is at least one sign change of $\Delta a(k)$ on each (finite or infinite) interval delimited by the locations of sign changes of $a(k)$.


\subsection{Absolutely monotone-then-completely monotone sequences}
\label{sec:amcm}

Recall that a one-sided sequence $a(k)$, $k \ge 0$, is \emph{completely monotone}, or $\cm$, if $(-1)^n \Delta^n a(k) \ge 0$ for every $n \ge 0$ and $k \ge 0$. A reversed one-sided sequence $a(k)$ defined for $k \le 0$ is called \emph{absolutely monotone}, or $\am$, if the sequence $a(-k)$, $k \ge 0$, is completely monotone. Equivalently, $a(k)$ is absolutely monotone if $\Delta^n a(k) \ge 0$ for every $n \ge 0$ and $k \le -n$.

The celebrated Hausdorff's theorem identifies completely monotone sequences with moment sequences of measures on $[0, 1]$: a sequence $a(k)$, $k \ge 0$, is completely monotone if and only if there exists a finite measure $\mu$ on $[0, 1]$ such that
\formula{
 a(k) & = \int_{[0, 1]} s^k \mu(ds) .
}
Furthermore, by the dominated convergence theorem, $a(k)$ converges to zero if and only if $\mu(\{1\}) = 0$. By Fubini's theorem, the generating function of a $\cm$ sequence is given by
\formula[eq:cm:gf]{
 F(z) & = \sum_{k = 0}^\infty a(k) z^k = \int_{[0, 1]} \frac{1}{1 - s z} \, \mu(ds)
}
in the open unit disc $|z| < 1$. If $a(k)$ converges to zero as $k \to \infty$, then formula~\eqref{eq:cm:gf} holds also when $|z| = 1$ and $z \ne 1$. By Fubini's theorem, formula~\eqref{eq:cm:gf} also holds when $z = 1$, with both sides possibly infinite. We conclude that if $a(k)$ is $\cm$ and it converges to zero, then~\eqref{eq:cm:gf} holds in the closed unit disc $|z| \le 1$. Observe that the right-hand side of~\eqref{eq:cm:gf} in fact defines a holomorphic function of $z \in \C \setminus [1, \infty)$.

Following the analogous definition of $\amcm$ functions in~\cite{kwasnicki}, in this section we introduce the class of $\amcm$ sequences.

\begin{definition}
A two-sided sequence $a(k)$ is said to be \emph{absolutely monotone-then-completely monotone}, or $\amcm$ in short, if it is not identically zero, the one-sided seuqence $a(k)$, $k \ge 0$, is completely monotone, while the reversed one-sided sequence $a(k)$, $k \le 0$, is absolutely monotone.
\end{definition}

Equivalently, $a(k)$ is $\amcm$ if and only if both one-sided sequences: $a(k)$, $k \ge 0$, and $a(-k)$, $k \ge 0$, are completely monotone. By Hausdorff's characterisation of completely monotone sequences, a two-sided sequence $a(k)$ is $\amcm$ if and only if there are finite measures $\mu_+$, $\mu_-$ such that
\formula{
 a(k) & = \int_{[0, 1]} s^k \mu_+(ds) && \text{for $k \ge 0$,} \\
 a(k) & = \int_{[0, 1]} s^{-k} \mu_-(ds) && \text{for $k \le 0$,}
}
and by comparing the two conflicting definitions of $a(0)$ we see that necessarily $\mu_+([0, 1]) = \mu_-([0, 1])$. Furthermore, $a(k)$ converges to zero as $k \to \pm\infty$ if and only if $\mu_+(\{1\}) = \mu_-(\{1\}) = 0$, and in this case, by~\eqref{eq:cm:gf}, the generating function of $a(k)$ is given by
\formula{
 F(z) = \sum_{k = -\infty}^\infty a(k) z^k & = \sum_{k = 0}^\infty a(k) z^k + \sum_{k = 0}^\infty a(-k) z^{-k} - a(0) \\
 & = \int_{[0, 1)} \frac{1}{1 - s z} \, \mu_+(ds) + \int_{[0, 1)} \frac{1}{1 - s / z} \, \mu_-(ds) - a(0)
}
whenever $|z| = 1$, $z \ne 1$. Since $a(0) = \int_{[0, 1)} \mu_+(ds)$, the above expression simplifies to
\formula*[eq:amcm:gf]{
 F(z) & = \int_{(0, 1)} \frac{s z}{1 - s z} \, \mu_+(ds) + \int_{[0, 1)} \frac{1}{1 - s / z} \, \mu_-(ds) .
}
Conversely, if $\mu_+$ and $\mu_-$ are finite measures on $[0, 1)$ such that $\mu_+([0, 1)) = \mu_-([0, 1))$, then the above formula defines the generating function $F$ of a unique $\amcm$ sequence $a(k)$ which converges to zero as $k \to \pm\infty$.

Observe that the right-hand side of~\eqref{eq:amcm:gf} defines a holomorphic function of $z \in \C \setminus [0,\infty)$, and we use the same symbol $F(z)$ to denote this holomorphic extension. The main result of this section describes the exponential representation of the function $F(z)$.

\begin{lemma}
\label{lem:amcm}
If $F$ is the generating function of an $\amcm$ sequence $a(k)$ which converges to zero as $k \to \pm\infty$, then $F$ extends to a holomorphic function on $\C \setminus [0, \infty)$, given by
\formula[eq:amcm:gf:exp]{
 F(z) & = \exp\biggl(c + \int_0^\infty \biggl(\frac{1}{s - z} - \frac{s}{s^2 + 1}\biggr) \ph(s) ds\biggr) ,
}
and we have
\formula[eq:amcm:gf:exp:lim]{
 \lim_{t \to 0} (e^{i t} - 1) F(e^{i t}) & = 0 .
}
Here $c \in \R$ and $\ph$ is a Borel function on $(0, \infty)$ which takes values in $[-1, 0]$ over $(0, 1)$ and in $[0, 1]$ over $(1, \infty)$. If we define $\ph(s) = 0$ for $s \le 0$, then
\formula[eq:amcm:gf:exp:phi]{
 \ph(s) & = \frac{1}{\pi} \lim_{t \to 0^+} \Arg F(s + i t)
}
for almost every $s \in \R$.

Conversely, if $F$ is given by~\eqref{eq:amcm:gf:exp} and all conditions listed above are satisfied, then $F$ is the generating function of a unique $\amcm$ sequence $a(k)$ which converges to zero as $k \to \pm\infty$.
\end{lemma}

We split the proof into two parts.

\begin{proof}[Proof of the direct part]
Suppose that $F$ is the generating function of an $\amcm$ sequence $a(k)$ which converges to zero as $k \to \pm\infty$. We divide the proof into three steps.

\emph{Step 1.} Recall that $(z - 1) F(z)$ is the generating function of the sequence $\Delta a(k)$, which is summable, and which sums up to zero. Thus, $(z - 1) F(z)$ extends continuously to the unit circle $|z| = 1$, and it takes value zero at $z = 1$. As in Remark~\ref{rem:c5}, this proves~\eqref{eq:amcm:gf:exp:lim}.

\emph{Step 2.} We already know that $F$ is given by~\eqref{eq:amcm:gf}. In order to derive~\eqref{eq:amcm:gf:exp}, we first claim that $G(z) = (1 - 1 / z) F(z)$, which is the generating function of the sequence $\Delta a(k - 1)$, is a Pick function. By~\eqref{eq:amcm:gf}, we have
\formula*[eq:amcm:gf:aux:1]{
 G(z) & = \int_{(0, 1)} \frac{s z - s}{1 - s z} \, \mu_+(ds) + \int_{[0, 1)} \frac{1 - 1 / z}{1 - s / z} \, \mu_-(ds)
}
Observe that
\formula{
 \frac{s z - s}{1 - s z} & = \frac{s (1 - s)}{s^2 + 1} \, \frac{1 + z / s}{1 / s - z} - \frac{s (s + 1)}{s^2 + 1}
}
and
\formula{
 \frac{1 - 1 / z}{1 - s / z} & = \frac{1 - s}{s^2 + 1} \, \frac{1 + s z}{s - z} + \frac{s + 1}{s^2 + 1} .
}
Thus, \eqref{eq:amcm:gf:aux:1} can be rewritten as in~\eqref{eq:pick:alt}:
\formula[eq:amcm:gf:aux:2]{
 G(z) & = d + \int_{[0, \infty)} \frac{1 + s z}{s - z} \, \sigma(ds) ,
}
where
\formula{
 \sigma(A) & = \int_{(0, 1)} \ind_A(1 / s) \, \frac{s (1 - s)}{s^2 + 1} \, \mu_+(ds) + \int_{[0, 1)} \ind_A(s) \, \frac{1 - s}{s^2 + 1} \, \mu_-(ds)
}
and
\formula{
 d & = -\int_{(0, 1)} \frac{s (s + 1)}{s^2 + 1} \, \mu_+(ds) + \int_{[0, 1)} \frac{s + 1}{s^2 + 1} \, \mu_-(ds) .
}
In particular, $G$ is indeed a Pick function, and our claim is proved. We also note that $\sigma(\{1\}) = 0$.

\emph{Step 3.} The exponential representation~\eqref{eq:pick:exp} of $G$ reads
\formula[eq:amcm:gf:aux:3]{
 G(z) & = \exp\biggl(\tilde{c} + \int_{-\infty}^\infty \biggl(\frac{1}{s - z} - \frac{s}{s^2 + 1}\biggr) \tilde{\ph}(s) ds\biggr)
}
when $\im z > 0$, where $\tilde{c} \in \R$ and $\tilde{\ph}$ is a Borel function on $\R$ taking values in $[0, 1]$. Furthermore,
\formula{
 \tilde{\ph}(s) & = \frac{1}{\pi} \lim_{t \to 0^+} \Arg G(s + i t)
}
for almost every $s > 0$. By~\eqref{eq:amcm:gf}, when $z < 0$ we have
\formula{
 F(z) & = -\int_{(0, 1)} \frac{s |z|}{1 + s |z|} \, \mu_+(ds) + \int_{[0, 1)} \frac{1}{1 + s / |z|} \, \mu_-(ds) \\
 & > -\frac{|z|}{1 + |z|} \, \mu_+((0, 1)) + \frac{1}{1 + 1 / |z|} \, \mu_-([0, 1)) \\
 & = \frac{|z|}{1 + |z|} \, (\mu_-([0, 1)) - \mu_+((0, 1))) \ge 0
}
(recall that $\mu_+([0, 1)) = \mu_-([0, 1)) = a(0)$), and hence $G(z) = (1 - 1 / z) F(z) > 0$ when $z < 0$. Since $G$ is continuous on $\C \setminus [0, \infty)$, the above formula for $\tilde{\ph}(s)$ implies that $\tilde{\ph}(s) = 0$ for $s < 0$. Let us define
\formula{
 \ph(s) & = \tilde{\ph}(s) - \ind_{(0, 1)}(s) ,
}
so that $\ph(s) = 0$ for $s < 0$, $\ph(s) \in [-1, 0]$ for $s \in (0, 1)$, and $\ph(s) \in [0, 1]$ for $s \in (1, \infty)$. With this definition, we find that
\formula{
 G(z) & = \exp\biggl(\tilde{c} + \int_0^1 \biggl(\frac{1}{s - z} - \frac{s}{s^2 + 1}\biggr) ds + \int_0^\infty \biggl(\frac{1}{s - z} - \frac{s}{s^2 + 1}\biggr) \ph(s) ds\biggr)
}
when $\im z > 0$. By a straightforward calculation, the first integral on the right-hand side is equal to $\log(1 - 1 / z) - \tfrac{1}{2} \log 2$, and thus
\formula{
 G(z) & = (1 - 1 / z) \exp\biggl(c + \int_0^\infty \biggl(\frac{1}{s - z} - \frac{s}{s^2 + 1}\biggr) \ph(s) ds\biggr) ,
}
where $c = \tilde{c} - \tfrac{1}{2} \log 2$. This proves~\eqref{eq:amcm:gf:exp}. Finally, we have
\formula{
 \frac{1}{\pi} \lim_{t \to 0^+} \Arg F(s + i t) & = \frac{1}{\pi} \lim_{t \to 0^+} \biggl(\Arg G(s + i t) - \Arg\biggl(1 - \frac{1}{s + i t}\biggr)\biggr) \\
 & = \tilde{\ph}(s) - \ind_{(0, 1)}(s) = \ph(s)
}
for almost every $s \in \R$. Formula~\eqref{eq:amcm:gf:exp:phi} follows, and the proof is complete.
\end{proof}

\begin{proof}[Proof of the converse part]
Suppose that $\ph$ is a Borel function on $(0, \infty)$ which takes values in $[-1, 0]$ on $(0, 1)$ and in $[0, 1]$ on $(1, \infty)$, $c \in \R$, $F$ is given by~\eqref{eq:amcm:gf:exp}, and~\eqref{eq:amcm:gf:exp:lim} holds. We essentially reverse the argument from the proof of the direct part.

\emph{Step 1.} Observe that $F$ is a holomorphic function in $\C \setminus [0, \infty)$, and $F(z) > 0$ for $z < 0$. Reversing Step~3 of the proof of the direct part, we find that $G(z) = (1 - 1 / z) F(z)$ is a Pick function, with exponential representation~\eqref{eq:amcm:gf:aux:3}, where $\tilde{\ph}(s) = \ph(s) + \ind_{(0, 1)}(s)$ and $\tilde{c} \in \R$.

\emph{Step 2.} The Stieltjes representation~\eqref{eq:pick:alt} of the Pick function $G$ reads
\formula{
 G(z) & = d + \int_{\R \cup \{\infty\}} \frac{1 + s z}{s - z} \, \sigma(ds) ,
}
where $d \in \R$ and $\sigma$ is a finite measure on $\R \cup \{\infty\}$. Since $G$ is a holomorphic function in $\C \setminus [0, \infty)$, $\sigma$ necessarily vanishes on $(-\infty, 0)$. 

Recall that the integrand on the right-hand side is understood to be equal to $z$ for $s = \infty$. For $s > 0$ and $z < -1$ we have
\formula{
 \biggl|\frac{1 + s z}{z (s - z)}\biggr| & \le \frac{1 + s |z|}{|z| (s + |z|)} \le 1 .
}
Hence, by the dominated convergence theorem, we have
\formula{
 \lim_{z \to -\infty} \frac{G(z)}{z} & = \lim_{z \to -\infty} \int_{[0, \infty) \cup \{\infty\}} \frac{1 + s z}{z (s - z)} \, \sigma(ds) = \sigma(\{\infty\}) .
}
Since $F(z) > 0$ for $z < 0$, we have $G(z) = (1 - 1 / z) F(z) > 0$ for $z < 0$, and so the left-hand side is nonpositive. Therefore, $\sigma(\{\infty\}) = 0$. In other words, $G$ has the representation~\eqref{eq:amcm:gf:aux:2}.

Reversing the argument from Step~2 of the proof of the direct part, we find that instead of~\eqref{eq:amcm:gf:aux:1}, we have
\formula{
 G(z) & = \tilde{d} + \sigma(\{1\}) \, \frac{1 + z}{1 - z} + \int_{(0, 1)} \frac{s z - s}{1 - s z} \, \mu_+(ds) + \int_{[0, 1)} \frac{1 - 1 / z}{1 - s / z} \, \mu_-(ds)
}
for some $\tilde{d} \in \R$ and some finite measures $\mu_+$ on $(0, 1)$ and $\mu_-$ on $[0, 1)$. However, by~\eqref{eq:amcm:gf:exp:lim}, $G(e^{i t}) = e^{-i t} (e^{i t} - 1) F(e^{i t})$ converges to zero as $t \to 0$. Using the above expression for $G(e^{i t})$, the dominated convergence theorem, and the estimates
\formula{
 \biggl|\frac{s e^{i t} - s}{1 - s e^{i t}}\biggr| & = \frac{s |e^{i t} - 1|}{|1 - s e^{i t}|} \le \frac{s t}{s \sin t} \le \frac{\pi}{2}
}
for $s \in (0, 1)$ and $t \in (0, \tfrac{\pi}{2})$, and
\formula{
 \biggl|\frac{1 - e^{-i t}}{1 - s e^{-i t}}\biggr| & = \biggl|1 - \frac{(1 - s) e^{-i t}}{1 - s e^{-i t}}\biggr| \le 1 + \frac{1 - s}{|1 - s e^{-i t}|} \le 2
}
for $s \in [0, 1)$ and $t \in (0, \tfrac{\pi}{2})$, we find that
\formula{
 0 & = \lim_{t \to 0^+} G(e^{i t}) = \tilde{d} + \sigma(\{1\}) i \infty .
}
Hence, $\tilde{d} = \sigma(\{1\}) = 0$. We conclude that in fact~\eqref{eq:amcm:gf:aux:1} holds with no modifications, and with $\sigma(\{1\}) = 0$. It follows that $F$ is given by~\eqref{eq:amcm:gf}:
\formula{
  F(z) & = \int_{(0, 1)} \frac{s z}{1 - s z} \, \mu_+(ds) + \int_{[0, 1)} \frac{1}{1 - s / z} \, \mu_-(ds) .
}

\emph{Step 3.} It remains to extend the definition of $\mu_+$ at $0$ so that $\mu_+([0, 1)) = \mu_-([0, 1))$. In other words, we let $\mu_+(\{0\}) = \mu_-([0, 1)) - \mu_+((0, 1))$. By the dominated convergence theorem,
\formula{
 \lim_{z \to -\infty} F(z) & = \lim_{z \to -\infty} \int_{(0, 1)} \frac{s z}{1 - s z} \, \mu_+(ds) + \lim_{z \to -\infty} \int_{[0, 1)} \frac{1}{1 - s / z} \, \mu_-(ds) \\
 & = -\mu_+((0, 1)) + \mu_-([0, 1)) = \mu_+(\{0\}) .
}
On the other hand, $F(z) > 0$ when $z < 0$, and hence $\mu_+(\{0\}) \ge 0$. By the observation made after equation~\eqref{eq:amcm:gf}, $F(z)$ is the generating function of an $\amcm$ sequence $a(k)$ which converges to zero as $k \to \pm\infty$.
\end{proof}

Lemma~\ref{lem:amcm} is an extension of the results of Section~2.7 in~\cite{kw}, where one-sided completely monotone sequences were studied. It is closely related to the results of Section~3 in~\cite{kwasnicki}, where $\amcm$ functions were introduced.


\subsection{Pólya frequency sequences}
\label{sec:pff}

A two-sided summable sequence $a(k)$ is a \emph{Pólya frequency sequence} if it is not identically zero, and the doubly infinite matrix $(a(k - l) : k, l \in \Z)$ is totally positive, that is, all of its minors are nonnegative. Summable Pólya frequency sequences can be characterised by their generating functions, which are necessarily of the form
\formula[eq:pf:gf]{
 F(z) & = \sum_{k = -\infty}^\infty a(k) z^k = z^m \exp\biggl(b^+ z + \frac{b^-}{z} + c\biggr) \prod_{j = 1}^\infty \frac{(1 + \gamma^+_j z)(1 + \gamma^-_j / z)}{(1 - \delta^+_j z)(1 - \delta^-_j / z)} \, ,
}
where $m$ is an integer, $b^+, b^- \in [0, \infty)$, $c \in \R$, and $\gamma^\pm_j, \delta^\pm_j$ are nonnegative summable sequences with $\gamma^\pm_j \le 1$ and $\delta^\pm_j < 1$; see Theorem~8.9.5 in~\cite{karlin}.

Another characterisation of summable two-sided Pólya frequency sequences involves the variation diminishing property: a summable two-sided sequence $a(k)$ is, up to sign, a Pólya frequency sequence if and only if the convolution with $a(k)$ does not increase the number of sign changes; see Theorem~5.1.5 in~\cite{karlin}. We will need this result (in fact, only its easy direct half) in the proof the implication \ref{thm:main:b}~$\Rightarrow$~\ref{thm:main:a} in Theorem~\ref{thm:main}.

We now derive the exponential representation of generating functions of summable two-sided Pólya frequency sequences. By~\eqref{eq:pf:gf}, we have
\formula{
 F(z) & = \exp\biggl(m \log z + b^+ z + \frac{b^-}{z} + c \\
 & \hspace*{5em} + \sum_{j = 1}^\infty \bigl( \log(1 + \gamma^+_j z) + \log(1 + \gamma^-_j / z) - \log(1 - \delta^+_j z) - \log(1 - \delta^-_j / z)\bigr)\biggr)
}
when $\im z > 0$. By a straightforward calculation,
\formula{
 \log z & = \int_{-\infty}^0 \biggl(\frac{1}{s - z} - \frac{s}{s^2 + 1}\biggr) ds , \displaybreak[0] \\
 \log (1 + \gamma^+_j z) & = \int_{-\infty}^{-1 / \gamma^+_j} \biggl(\frac{1}{s - z} - \frac{s}{s^2 + 1}\biggr) ds + \log \sqrt{(\gamma^+_j)^2 + 1} , \displaybreak[0] \\
 \log (1 + \gamma^-_j / z) & = -\int_{-\gamma^-_j}^0 \biggl( \frac{1}{s - z} - \frac{s}{s^2 + 1}\biggr) ds + \log \sqrt{(\gamma^-_j)^2 + 1}, \displaybreak[0] \\
 \log (1 - \delta^+_j z) & = -\int_{1 / \delta^+_j}^\infty \biggl(\frac{1}{s - z} - \frac{s}{s^2 + 1}\biggr) ds + \log \sqrt{(\delta^+_j)^2 + 1}, \displaybreak[0] \\
 \log (1 - \delta^-_j / z) & = \int_0^{\delta^-_j} \biggl( \frac{1}{s - z} - \frac{s}{s^2 + 1}\biggr) ds + \log \sqrt{(\delta^-_j)^2 + 1}.
}
It follows that
\formula{
 F(z) & = \exp\biggl(b^+ z + \frac{b^-}{z} + d + \int_{-\infty}^\infty \biggl(\frac{1}{s - z} - \frac{s}{s^2 + 1}\biggr) \ph(s) \, ds\biggr)
}
when $\im z > 0$, where
\formula{
 d & = c + \sum_{j = 1}^\infty \biggl(\log\sqrt{(\gamma^+_j)^2 + 1} + \log\sqrt{(\gamma^-_j)^2 + 1} - \log\sqrt{(\delta^+_j)^2 + 1} - \log\sqrt{(\delta^-_j)^2 + 1}\biggr)
}
(the series converges, because $\log \sqrt{s^2 + 1} \le s$ for $s \ge 0$) and
\formula{
 \ph(s) & = m \ind_{(-\infty, 0)}(s) + \sum_{j = 1}^\infty \biggl(\ind_{(-\infty, -1 / \gamma^+_j)}(s) - \ind_{(-\gamma^-_j, 0)}(s) - \ind_{(0, \delta^-_j)}(s) + \ind_{(1 / \delta^+_j, \infty)}(s)\biggr) .
}
In particular, $\ph$ is stepwise decreasing on $(-\infty, 0)$, stepwise increasing on $(0, \infty)$, and equal to zero near $1$. Finally, we note the sequences $\gamma^\pm_j$ and $\delta^\pm_j$ are summable if and only if
\formula{
 \int_{-\infty}^\infty \frac{|\ph(s)|}{s^2 + 1} \, ds & < \infty .
}
The above reasoning can be clearly reversed, so that we obtain an equivalent form of the generating function of a summable two-sided Pólya frequency sequence. We state this result as a lemma.

\begin{lemma}
\label{lem:pf}
If $F(z)$ is the generating function of a summable two-sided Pólya frequency sequence $a(k)$, then $F$ extends to a holomorphic function in $\C \setminus \R$, given by
\formula[eq:pf:gf:exp]{
 F(z) & = \exp\biggl(b^+ z + \frac{b^-}{z} + d + \int_{-\infty}^\infty \biggl(\frac{1}{s - z} - \frac{s}{s^2 + 1}\biggr) \ph(s) \, ds\biggr) .
}
Here $b^\pm \ge 0$, $d \in \R$, and $\ph$ is a stepwise decreasing function on $(-\infty, 0)$, a stepwise increasing function on $(0, \infty)$, $\ph(s) = 0$ in a neighbourhood of $1$, and
\formula[eq:pf:gf:exp:int]{
 \int_{-\infty}^\infty \frac{|\ph(s)|}{s^2 + 1} \, ds & < \infty .
}

Conversely, if $F$ is given by~\eqref{eq:pf:gf:exp} and all conditions listed above are satisfied, then $F$ is the generating function of a unique summable two-sided Pólya frequency sequence $a(k)$.
\end{lemma}

For a similar discussion of the one-sided case, see Section~2.8 in~\cite{kw}. Pólya frequency functions in a similar context were discussed in Section~4 in~\cite{kwasnicki}. 


\subsection{Generating functions and inversion formulae}
\label{sec:post}

The classical Post's inversion formula tells us that if $G$ is the Laplace transform of $F$:
\formula{
 G(x) & = \int_0^\infty e^{-t x} F(t) dt ,
}
then, under suitable assumptions on $F$, we have
\formula{
 F(t) & = \lim_{n \to \infty} \frac{(-1)^n (n / t)^{n + 1}}{n!} \, G^{(n)}(n / t) .
}
This identity, applied to the Fourier transform $F$ of a bell-shaped function, played a crucial role in the analysis of bell-shaped functions in~\cite{ks}.

For one-sided bell-shaped sequences, the following discrete counterpart of Post's inversion formula was used in~\cite{kw}. If $G$ is the moment sequence of a function $F$:
\formula{
 G(k) & = \int_0^1 t^k F(t) dt ,
}
then, again under suitable assumptions on $F$, we have
\formula{
 F(t) & = \lim_{n \to \infty} \frac{(-1)^n (k_n + 1)\rf{n + 1}}{n!} \, \Delta^n G(k_n) \, ,
}
where $\Delta$ is the forward difference operator and
\formula{
 \lim_{n \to \infty} \frac{k_n}{n + k_n} & = t .
}
In~\cite{kw} this formula was applied to the generating function $F$ of a one-sided bell-shaped sequence.

In our case, the generating function $F$ of a two-sided bell-shaped sequence is only defined on the unit circle. Therefore, we need a different variant of Post's inversion formula. We were not able to find this result in the literature, so we provide a complete proof.

\begin{proposition}[yet another Post's inversion formula]
\label{prop:post}
Let $F$ be an integrable function on the unit circle in the complex plane, and let
\formula{
 G(x) & = \int_0^\infty e^{i t x} F(e^{i t}) dt
}
when $\im x > 0$. Then for every $t \in (0, \pi)$ such that $F$ is continuous at $e^{i t}$ we have
\formula[eq:post]{
 F(e^{i t}) & = \lim_{n \to \infty} \frac{(-1)^n (x_n)\rf{n + 1}}{i n!} \, \Delta^n G(x_n) ,
}
where $\Delta$ is the forward difference operator and
\formula{
 x_n & = \frac{n}{2} \biggl(i \cot \frac{t}{2} - 1\biggr) .
}
The same result holds true in the more general case when instead of integrability of $F$ we assume that $(e^{i t} - 1)^n F(e^{i t})$ is integrable over $(0, 2 \pi)$ for sufficiently large $n$, provided that in~\eqref{eq:post} we agree that
\formula{
 \Delta^n G(x) & = \int_0^\infty e^{i t x} (e^{i t} - 1)^n F(e^{i t}) dt
}
for $n$ large enough.
\end{proposition}

\begin{proof}
We divide the argument into four steps.

\emph{Step 1.} Suppose that $G$ is integrable over the unit circle. The $n$th iterated difference of $x \mapsto e^{i s x}$ is equal to $e^{i s x} (e^{i s} - 1)^n$. Thus,
\formula{
 \Delta^n G(x) & = \int_0^\infty e^{i s x} (e^{i s} - 1)^n F(e^{i s}) ds \\
 & = (2 i)^n \int_0^\infty e^{i s (x + n/2)} \biggl(\sin \frac{s}{2}\biggr)^n F(e^{i s}) ds
}
when $\im x > 0$. We fix $t \in (0, \pi)$ and we let $c = \frac{1}{2} \cot \frac{t}{2} \in (0, \infty)$, so that $x_n = -\frac{n}{2} + n i c$. The above formula evaluated at $x = x_n$ leads to
\formula*[eq:post:delta]{
 \Delta^n G(x_n) & = (2 i)^n \int_0^\infty \biggl(e^{-c s} \sin \frac{s}{2}\biggr)^n F(e^{i s}) ds .
}

\emph{Step 2.} Suppose for the moment that $F$ is constant $1$. Then $G(x) = i / x$, and hence, by induction,
\formula{
 \Delta^n G(x) & = \frac{(-1)^n i n!}{x\rf{n + 1}} \, .
}
On the other hand, $\Delta^n G(x)$ is given by~\eqref{eq:post:delta}. It follows that
\formula{
 (2 i)^n \int_0^\infty \biggl(e^{-c s} \sin \frac{s}{2}\biggr)^n ds & = \frac{(-1)^n i n!}{(x_n)\rf{n + 1}} \, .
}
In other words, if we write
\formula{
 M_n & = \int_0^\infty \biggl(e^{-c s} \sin \frac{s}{2}\biggr)^n ds ,
}
then
\formula{
 (2 i)^n M_n & = \frac{(-1)^n i n!}{(x_n)\rf{n + 1}} \, .
}
Turning back to an arbitrary integrable function $F$, and combining the above expression with~\eqref{eq:post:delta}, we find that
\formula{
 \frac{(-1)^n (x_n)\rf{n + 1}}{i n!} \, \Delta^n G(x_n) & = \frac{1}{M_n} \int_0^\infty \biggl(e^{-c s} \sin \frac{s}{2}\biggr)^n F(e^{i s}) ds .
}

\emph{Step 3.} The function $|e^{-c s} \sin \frac{s}{2}|$, defined for $s \in (0, \infty)$, has a strict global maximum at $s = 2 \arccot(2 c) = t$. Using a standard `approximation to the identity' argument, one can show that if $F$ is continuous at $e^{i t}$, then
\formula{
 \lim_{n \to \infty} \frac{1}{M_n} \int_0^\infty \biggl(e^{-c s} \sin \frac{s}{2}\biggr)^n F(e^{i s}) ds & = F(e^{i t}) ,
}
and the desired result follows. For completeness, we provide full details in Lemma~\ref{lem:approximation} in Appendix~\ref{app:a}.

\emph{Step 4.} If we only assume that $(z - 1)^n F(z)$ is integrable over the unit circle for $n$ large enough, then the first equality in Step~1 holds now by assumption when $n$ is large enough. Otherwise, the proof is exactly the same.
\end{proof}


\subsection{Completeness of a class of holomorphic functions}
\label{sec:compact}

Below we prove that the class of functions $F$ given by~\eqref{eq:main} is closed under pointwise limits over the semi-circle $|z| = 1$, $\im z > 0$. This is very similar to various analogous results for other classes of functions, but we failed to find the statement needed here in literature.

It is well-known that pointwise convergence of Pick functions on a sufficiently large subset of the upper complex half-plane is equivalent to locally uniform convergence, as well as to the vague convergence of the corresponding Stieltjes measures. A brief discussion and some references can be found in Section~2.6 in~\cite{kw}. In that paper this property of Pick functions easily led to a similar completeness result for the class of generating functions of one-sided bell-shaped sequences; see Step~10 of the proof of Theorem~3.1 in~\cite{kw}. However, unlike in the one-sided case, logarithms of generating functions of two-sided bell-shaped sequences are generally not Pick functions. Thus, a different argument is needed.

A similar completeness result for the class of Fourier transforms of two-sided bell-shaped functions is given in Lemma~3.1 in~\cite{ks}. To some extent, the proof given below for two-sided bell-shaped sequences is similar to the argument used in~\cite{ks}. Our case is, however, more complicated, as the representing measures $\sigma_n$ change sign twice, and one of this sign changes occurs at an undetermined position.

\begin{lemma}
\label{lem:compact}
Suppose that
\formula[eq:compact:gfn]{
 F_n(z) & = \exp\biggl(b_n^+ z + \frac{b_n^-}{z} + c_n + \int_{-\infty}^\infty \biggl(\frac{1}{s - z} - \frac{s}{s^2 + 1}\biggr) \ph_n(s) ds \biggr) ,
}
where $b_n^\pm, c_n, \ph_n$ satisfy the conditions listed in Theorem~\ref{thm:main}\ref{thm:main:c}, except possibly~\ref{thm:main:c5}. Suppose, furthermore, that $F_n(e^{i t})$ converges as $n \to \infty$ to a finite limit $F(e^{i t})$ for every $t \in (0, \pi)$. Then either $F(e^{i t}) = 0$ for every $t \in (0, \pi)$ or $F$ extends to a holomorphic function in the upper complex half-plane, given by
\formula[eq:compact:gf]{
 F(z) & = \exp\biggl(b^+ z + \frac{b^-}{z} + c + \int_{-\infty}^\infty \biggl(\frac{1}{s - z} - \frac{s}{s^2 + 1}\biggr) \ph(s) ds \biggr) ,
}
where $b^\pm, c, \ph$ satisfy the conditions listed in Theorem~\ref{thm:main}\ref{thm:main:c}, except possibly~\ref{thm:main:c5}. Furthermore, in the latter case, as $n \to \infty$, the sequence $c_n$ converges to $c$, and the signed measures
\formula{
 \sigma_n(ds) & = b_n^+ \delta_\infty(ds) - b_n^- \delta_0(ds) + \frac{\ph_n(s)}{s^2 + 1} \, ds
}
(see~\eqref{eq:main:alt}) converge vaguely on $\R \cup \{\infty\}$ to the corresponding signed measure
\formula{
 \sigma(ds) & = b^+ \delta_\infty(ds) - b^- \delta_0(ds) + \frac{\ph(s)}{s^2 + 1} \, ds .
}
\end{lemma}

\begin{proof}
Assume that $F(e^{i t})$ is not identically equal to zero. We divide the argument into six steps.

\emph{Step 1.} We denote $k_n = -\ph_n(-1)$. Note that since $\ph_n$ only takes integer values on $(-\infty, 0)$, $k_n$ is an integer. We consider the function
\formula{
 G_n(z) & = z^{k_n} F_n(z) .
}
Clearly, $|G_n(e^{i t})| = |F_n(e^{i t})|$ for $t \in (0, \pi)$. When $\im z > 0$, we have
\formula[eq:compact:log]{
 \int_{-\infty}^0 \biggl(\frac{1}{s - z} - \frac{s}{s^2 + 1}\biggr) ds & = \log z ,
}
and therefore, by~\eqref{eq:compact:gfn},
\formula{
 G_n(z) & = \exp(k_n \log z) F_n(z) = \exp\biggl(b_n^+ z + \frac{b_n^-}{z} + c_n + \int_{-\infty}^\infty \biggl(\frac{1}{s - z} - \frac{s}{s^2 + 1}\biggr) \psi_n(s) ds \biggr) ,
}
where
\formula{
 \psi_n(s) & = \ph_n(s) + k_n \ind_{(-\infty, 0)}(s) .
}
Since $\ph_n$ is stepwise decreasing on $(-\infty, 0)$ and $\psi_n(s) = \ph_n(s) - \ph_n(-1)$ for $s < 0$, we have $\psi_n \ge 0$ on $(-\infty, -1) \cup (1, \infty)$ and $\psi_n \le 0$ on $(-1, 1)$. Define accordingly
\formula{
 \ro_n(ds) & = b_n^+ \delta_\infty(ds) - b_n^- \delta_0(ds) + \frac{\psi_n(s)}{s^2 + 1} \, ds \\
 & = \sigma_n(ds) + \frac{k_n}{s^2 + 1} \, \ind_{(-\infty, 0)}(s) ds .
}
Thus, $\ro_n$ is a finite signed measure, $\ro_n$ is nonnegative on $(\R \cup \{\infty\}) \setminus [-1, 1]$ and nonpositive on $[-1, 1]$, and $\ro_n(\{-1, 1\}) = 0$. The number 
\formula*[eq:compact:norm]{
 M_n & = \int_{\R \cup \{\infty\}} \frac{s^2 - 1}{s^2 + 1} \, \ro_n(ds) \\
 & = b_n^+ + b_n^- + \int_{-\infty}^\infty \frac{s^2 - 1}{(s^2 + 1)^2} \, \psi_n(s) ds ,
}
with the former integrand extended continuously at $s = \infty$, is therefore nonnegative and finite.

Rewriting the expression for $G_n$ as in~\eqref{eq:main:alt}, we obtain
\formula[eq:compact:shifted]{
 G_n(z) & = \exp\biggl(c_n + \int_{\R \cup \{\infty\}} \frac{1 + s z}{s - z} \, \ro_n(ds) \biggr) .
}
In the next step we study the properties of $|G_n(z)|$ when $z = e^{i t}$, $t \in (0, \pi)$: we prove a variant of uniform continuity of these functions.

\emph{Step 2.} For $t \in (0, \pi)$ the expression for $G_n(e^{i t})$ takes form
\formula{
 G_n(e^{i t}) & = \exp\biggl(c_n + \int_{\R \cup \{\infty\}} \frac{(s^2 - 1) \cos t + i (s^2 + 1) \sin t}{s^2 - 2 s \cos t + 1} \, \ro_n(ds) \biggr) ,
}
with the integrand extended continuously at $s = \infty$. In particular, by the dominated convergence theorem,
\formula{
 -\frac{d}{dt} \bigl(\log |G_n(e^{i t})|\bigr) & = \int_{\R \cup \{\infty\}} \frac{(s^2 - 1) (s^2 + 1) \sin t}{(s^2 - 2 s \cos t + 1)^2} \, \ro_n(ds) .
}
Observe that
\formula{
 (1 - |\cos t|) (s^2 + 1) & \le s^2 - 2 s \cos t + 1 \le 2 (s^2 + 1) ,
}
and hence
\formula[eq:compact:derivative]{
 \frac{\sin t}{4} \, M_n & \le -\frac{d}{dt} \bigl(\log |G_n(e^{i t})|\bigr) \le \frac{\sin t}{(1 - |\cos t|)^2} \, M_n
}
(see~\eqref{eq:compact:norm}). Finally, $|G_n(e^{i t})| = |F_n(e^{i t})|$, so that~\eqref{eq:compact:derivative} also holds with $G_n$ replaced by $F_n$.

\emph{Step 3.} We claim that $M_n$ is a bounded sequence. By assumption, $F(e^{i t_1}) \ne 0$ for some $t_1 \in (0, \pi)$. We choose an arbitrary $t_2 \in (0, t_1)$. By~\eqref{eq:compact:derivative} (for $F_n$ instead of $G_n$),
\formula{
 \log |F_n(e^{i t_2})| - \log |F_n(e^{i t_1})| & = -\int_{t_2}^{t_1} \frac{d}{dt} \bigl(\log |G_n(e^{i t})|\bigr) dt \\
 & \ge \int_{t_2}^{t_1} \frac{\sin t}{4} \, M_n dt = \frac{\cos t_2 - \cos t_1}{4} \, M_n .
}
As $n \to \infty$, the expression on the left-hand side has a limit $\log |F(e^{i t_2})| - \log |F(e^{i t_1})| \in [-\infty, \infty)$. Thus, the sequence $M_n$ is necessarily bounded. Our claim is proved.

Since $\psi_n$ is increasing-after-rounding on $(0, \infty)$, for $r \in [1, 2]$ we have
\formula{
 0 \le \psi_n(r) & \le \int_2^3 (1 + \psi_n(s)) ds \le 1 + \frac{25}{3} \int_2^3 \frac{s^2 - 1}{(s^2 + 1)^2} \, \psi_n(s) ds \le 1 + \frac{25 M_n}{3} \, .
}
Similarly, for $r \in [\tfrac{1}{2}, 1]$,
\formula{
 0 \le -\psi_n(r) & \le 2 \int_0^{1/2} (1 - \psi_n(s)) ds \le 1 + \frac{25}{12} \int_0^{1/2} \frac{s^2 - 1}{(s^2 + 1)^2} \, \psi_n(r) dr \le 1 + \frac{25 M_n}{12} \, .
}
It follows that $|\psi_n(r)| \le 1 + \tfrac{25}{3} M_n$ for $r \in [\tfrac{1}{2}, 2]$. Similarly, since $\psi_n$ is stepwise decreasing on $(-\infty, 0)$, the same argument shows that $|\psi_n(r)| \le \tfrac{25}{3} M_n$ for $r \in [-2, -\tfrac{1}{2}]$. Finally, for $s \in (\R \cup \{\infty\}) \setminus [-2, 2]$ we have
\formula{
 \frac{5}{3} \, \frac{s^2 - 1}{s^2 + 1} & \ge 1 ,
}
and similarly for $s \in (-\tfrac{1}{2}, \tfrac{1}{2})$, 
\formula{
 \frac{5}{3} \, \frac{s^2 - 1}{s^2 + 1} & \le -1 .
}
By combining the above estimates, we find that the total variation norm of $\ro_n$, denoted as $\|\ro_n\|$, satisfies
\formula{
 \|\ro_n\| & = \int_{\R \cup \{\infty\}) \setminus [-2, 2]} \ro_n(ds) + \int_{(-1/2, 1/2)} (-\ro_n)(ds) + \int_{-2}^{-1/2} \frac{|\psi_n(s)|}{1 + s^2} \, ds + \int_{1/2}^2 \frac{|\psi_n(s)|}{1 + s^2} \, ds \\
 & \le \frac{5}{3} \int_{\R \cup \{\infty\}) \setminus [-2, 2]} \frac{s^2 - 1}{s^2 + 1} \, \ro_n(ds) + \frac{5}{3} \int_{(-1/2, 1/2)} \frac{s^2 - 1}{s^2 + 1} \ro_n(ds) \\
 & \hspace*{3em} + \int_{-2}^{-1/2} \frac{1 + \tfrac{25}{3} M_n}{1 + s^2} \, ds + \int_{1/2}^2 \frac{1 + \tfrac{25}{3} M_n}{1 + s^2} \, ds \\
 & \le \tfrac{5}{3} M_n + 3 (1 + \tfrac{25}{3} M_n) = 3 + \tfrac{80}{3} M_n .
}
We conclude that $\|\ro_n\|$ is a bounded sequence.

\emph{Step 4.} We already know that the sequence $M_n$ is bounded, and that $F(e^{i t_1}) \ne 0$ for some $t_1 \in (0, \pi)$. Pick any $t_2 \in (0, \pi)$. By~\eqref{eq:compact:derivative} and the mean value theorem, the ratio
\formula{
 \frac{|F_n(e^{i t_1})|}{|F_n(e^{i t_2})|} & = \exp\bigl(\log |F_n(e^{i t_1})| - \log |F_n(e^{i t_2})| \bigr)
}
is bounded by a constant (which depends only on $t_1$, $t_2$ and the bound on $M_n$). Passing to the limit as $n \to \infty$, we find that $F(e^{i t_2}) \ne 0$.

In particular, $|F(i)| \ne 0$. However, by~\eqref{eq:compact:gfn},
\formula{
 |F(i)| & = \lim_{n \to \infty} |F_n(i)| = \lim_{n \to \infty} e^{c_n} .
}
Therefore, $c_n$ has a finite limit $c$.

\emph{Step 5.}
W have already proved that the sequence $c_n$ converges to some $c \in \R$, and that the finite signed measures $\ro_n$ on $\R \cup \{\infty\}$ have bounded total variation norms. In particular, each subsequence of $\ro_n$ has a vaguely convergent further subsequence.

Suppose that $\ro$ is a partial limit of $\ro_n$ in the sense of vague convergence of measures. Passing to the limit along the corresponding subsequence $n_j$ in the representation~\eqref{eq:compact:shifted} of $G_n(e^{i t})$, we find that
\formula{
 \lim_{j \to \infty} G_{n_j}(e^{i t}) & = \exp \biggl(c + \int_{\R \cup \{\infty\}} \frac{1 + s e^{i t}}{s - e^{i t}} \, \ro(ds)\biggr)
}
for $t \in (0, \pi)$. On the other hand,
\formula{
 \lim_{j \to \infty} \exp(-i k_{n_j} t) G_{n_j}(e^{i t}) & = \lim_{j \to \infty} F_{n_j}(e^{i t}) = F(e^{i t}) .
}
Therefore, $\exp(-i k_{n_j} t)$ converges pointwise for every $t \in (0, \pi)$. A standard argument shows that $k_{n_j}$ necessarily converges to a finite limit $k$, and we conclude that
\formula{
 F(e^{i t}) & = e^{-i k t} \lim_{j \to \infty} G_{n_j}(e^{i t}) \\
 & = \exp \biggl(-i k t + c + \int_{\R \cup \{\infty\}} \frac{1 + s e^{i t}}{s - e^{i t}} \, \ro(ds)\biggr)
}
for $t \in (0, \pi)$.

Clearly, $(1 + s^2) \ind_{\R}(s) \ro_{n_j}(ds)$ converges vaguely to $(1 + s^2) \ind_{\R}(s) \ro(ds)$ on $\R$ (but not necessarily on $\R \cup \{\infty\}$). Since $(1 + s^2) \ind_{\R}(s) \ro_{n_j}(ds)$ has a stepwise decreasing density function $\psi_n(s)$ on $(-\infty, 0)$, by Lemma~2.2 in~\cite{kw}, also the limiting measure $(1 + s^2) \ind_{\R}(s) \ro(ds)$ has a stepwise decreasing density function on $(-\infty, 0)$. The same argument shows that $(1 + s^2) \ind_{\R}(s) \ro(ds)$ has an increasing-after-rounding density function on $(0, \infty)$. We denote the density function of $(1 + s^2) \ind_{\R}(s) \ro(ds)$ on $\R \setminus \{0\}$ by $\psi$. Additionally, we set $b^+ = \ro(\{\infty\})$ and $b^- = -\ro(\{0\})$. Undoing the transformation that led to~\eqref{eq:main:alt} in Remark~\ref{rem:alt}, we find that our representation of $F(e^{i t})$ reads:
\formula{
 F(e^{i t}) & = \exp\biggl(-i k t + b^+ e^{i t} + \frac{b^-}{e^{i t}} + c + \int_{-\infty}^\infty \biggl(\frac{1}{s - e^{i t}} - \frac{s}{s^2 + 1}\biggr) \psi(s) ds \biggr)
}
for $t \in (0, \pi)$. If we write $\ph(s) = \psi(s) - k \ind_{(-\infty, 0)}(s) ds$, then, by~\eqref{eq:compact:log}, the above expression is equivalent to~\eqref{eq:compact:gf} with $z = e^{i t}$, and clearly the right-hand side of~\eqref{eq:compact:gf} defines a holomorphic function of $z$ in the upper complex half-plane $\im z > 0$. The first part of the lemma is thus proved.

\emph{Step 6.} Above we have shown that there is a number $c \in \R$ and a finite signed measure $\sigma$ on $\R \cup \{\infty\}$ such that
\formula[eq:compact:gf:alt]{
 F(z) & = \exp\biggl(c + \int_{\R \cup \{\infty\}} \frac{1 + s z}{s - z} \, \sigma(ds) \biggr)
}
when $|z| = 1$ and $\im z > 0$ (cf.~\eqref{eq:main:alt}). Here $c$ is the limit of $c_n$, and we have
\formula{
 \sigma(ds) & = \ro(ds) - \frac{k}{1 + s^2} \, \ind_{(-\infty, 0)}(s) ds ,
}
where $\ro$ is the vague limit of a subsequence of
\formula{
 \ro_n(ds) & = \sigma_n(ds) + \frac{k_n}{1 + s^2} \, \ind_{(-\infty, 0)}(ds) ,
}
and $k$ is the corresponding partial limit of $k_n$. It follows that $\sigma$ is the vague limit of the corresponding subsequence of $\sigma_n$. Additionally, since every subsequence of $\ro_n$ has a vaguely convergent further subsequence, the sequence $\sigma_n$ has the same property.

A standard argument shows that~\eqref{eq:compact:gf:alt} determines the pair $c, \sigma$ uniquely. Indeed: suppose that $\tilde c$ and $\tilde{\sigma}$ is another such pair. Then
\formula{
 1 & = \exp\biggl((c - \tilde{c}) + \int_{\R \cup \{\infty\}} \frac{1 + s z}{s - z} \, (\sigma - \tilde{\sigma})(ds) \biggr) ,
}
and hence, for some integer $m$, we have
\formula{
 (c - \tilde{c}) + \int_{\R \cup \{\infty\}} \frac{1 + s z}{s - z} \, (\sigma - \tilde{\sigma})(ds) & = 2 m \pi i
}
when $|z| = 1$ and $\im z > 0$. The left-hand side defines a holomorphic function in the upper complex half-plane, and this function is necessarily constant. By uniqueness of the Cauchy--Stieltjes transform (see, for example, Theorem~II.1 in~\cite{donoghue}), $\sigma - \tilde{\sigma}$ is necessarily a zero measure, and $c - \tilde{c} = 2 m \pi i$. However, $c - \tilde{c}$ is real, and we conclude that $c - \tilde{c} = 0$.

We have thus proved that every subsequence of $\sigma_n$ has a vaguely convergent further subsequence, and the vague limit of this subsequence is necessarily the measure $\sigma$ constructed above. This, however, means that $\sigma_n$ converges vaguely to $\sigma$, and the proof is complete.
\end{proof}

%
%

\section{Proof of the main result}
\label{sec:proof}

This section is devoted to the proof of Theorem~\ref{thm:main}. It is divided into three parts, which correspond to three implications in the theorem, respectively: \ref{thm:main:b}~$\Rightarrow$~\ref{thm:main:a}, \ref{thm:main:a}~$\Rightarrow$~\ref{thm:main:c}, and \ref{thm:main:c}~$\Rightarrow$~\ref{thm:main:b}. For clarify, below we state these implications as three separate theorems.


\subsection{Convolutions of $\amcm$ and Pólya frequency sequences and their generating functions}
\label{sec:easy}

We begin with the easy parts of Theorem~\ref{thm:main}. The following result covers implication \ref{thm:main:b}~$\Rightarrow$~\ref{thm:main:a} in Theorem~\ref{thm:main}.

\begin{theorem}
\label{thm:amcm:pf}
Suppose that $a(k)$ is the convolution of a summable two-sided Pólya frequency sequence $b(k)$ and an $\amcm$ sequence $c(k)$ which converges to zero as $k \to \pm \infty$. Then $a(k)$ is bell-shaped.
\end{theorem}

\begin{proof}
The convolution $a(k) = (b * c)(k)$ is nonnegative, not identically zero, and it converges to zero as $k \to \pm \infty$ by the dominated convergence theorem. Thus, our goal is to prove that the sequence $\Delta^n a(k)$ changes its sign exactly $n$ times for $n = 0, 1, 2, \ldots$

By the discrete Rolle's theorem and induction, the sequence $\Delta^n a(k)$ changes sign at least $n$ times. To prove the converse inequality, we observe that by the definition of an $\amcm$ sequence we have $(-1)^n \Delta^n c(k) \ge 0$ for $k \ge 0$ and $\Delta^n c(k) \ge 0$ for $n \le -k$. Hence, $\Delta^n c(k)$ changes sign at most $n$ times, at positions $\alpha_0 = -n + 1$, $\alpha_1 = -n + 2$, \ldots, $\alpha_{n - 2} = -1$ and $\alpha_{n - 1} \ge 0$. The variation diminishing property of summable Pólya frequency sequences implies that also the convolution of $b(k)$ and $\Delta^n c(k)$ changes sign no more than $n$ times. It remains to observe that $(b * \Delta^n c)(k) = \Delta^n(b * c)(k)$.
\end{proof}

We now turn to implication \ref{thm:main:c}~$\Rightarrow$~\ref{thm:main:b} in Theorem~\ref{thm:main}. Our result in fact proves equivalence of these conditions. First, however, we need an auxiliary lemma.

\begin{lemma}
\label{lem:lim}
Let $p, q \ge 0$. Suppose that
\formula{
 F(z) & = \exp\biggl(b^+ z + \frac{b^-}{z} + c + \int_{-\infty}^\infty \biggl(\frac{1}{s - z} - \frac{s}{s^2 + 1}\biggr) \ph(s) ds \biggr)
}
when $|z| = 1$, $z \ne 1$, where $b^+, b^- \in [0, \infty)$, $c \in \R$ and $\ph$ is a Borel function on $\R$ such that $\ph(s) \le -p$ for $s \in (0, 1)$, $\ph(s) \ge q$ for $s \in (1, \infty)$, and $\int_{-\infty}^\infty |\ph(s)| / (s^2 + 1) ds < \infty$. If
\formula{
 \lim_{t \to 0^+} (e^{i t} - 1) F(e^{i t}) & = 0 ,
}
then $p + q < 1$.
\end{lemma}

\begin{proof}
When $t \in (0, \tfrac{\pi}{2})$, we have
\formula{
 |F(e^{i t})| & = \exp\biggl((b^+ + b^-) \cos t + c + \int_{-\infty}^\infty \re \biggl(\frac{1}{s - e^{i t}} - \frac{s}{s^2 + 1}\biggr) \ph(s) ds \biggr) .
}
Since
\formula{
 \re \biggl(\frac{1}{s - e^{i t}} - \frac{s}{s^2 + 1}\biggr) & = \frac{(s^2 - 1) \cos t}{(s^2 + 1) (s^2 - 2 s \cos t + 1)}
}
has the same sign as $\ph(s)$ for $s \in (0, \infty)$, we have
\formula{
 |F(e^{i t})| & \ge \exp\biggl((b^+ + b^-) \cos t + c + \int_{-\infty}^0 \re \biggl(\frac{1}{s - e^{i t}} - \frac{s}{s^2 + 1}\biggr) \ph(s) ds \\
 & \hspace*{3em} - p \int_0^1 \re \biggl(\frac{1}{s - e^{i t}} - \frac{s}{s^2 + 1}\biggr) ds \\
 & \hspace*{3em} + q \int_1^\infty \re \biggl(\frac{1}{s - e^{i t}} - \frac{s}{s^2 + 1}\biggr) ds\biggr) .
}
By a straightforward calculation, we obtain
\formula{
 |F(e^{i t})| & \ge 2^{p/2 + q/2} |e^{i t} - 1|^{-p - q} \times \\
 & \hspace*{3em} \times \exp\biggl((b^+ + b^-) \cos t + c + \int_{-\infty}^0 \re \biggl(\frac{1}{s - e^{i t}} - \frac{s}{s^2 + 1}\biggr) \ph(s) ds \biggr) .
}
By the dominated convergence theorem, as $t \to 0^+$, the exponent in brackets has a finite limit
\formula{
 b^+ + b^- + c + \int_{-\infty}^0 \re \biggl(\frac{1}{s - 1} - \frac{s}{s^2 + 1}\biggr) \ph(s) ds .
}
Therefore, if $(e^{i t} - 1) F(e^{i t})$ converges to zero as $t \to 0^+$, then $|e^{i t} - 1|^{1 - p - q}$ also converges to zero, and thus $p + q < 1$.
\end{proof}

\begin{theorem}
\label{thm:gf}
Suppose that $a(k)$ is the convolution of a summable two-sided Pólya frequency sequence $b(k)$ and an $\amcm$ sequence $c(k)$ which converges to zero as $k \to \pm \infty$. Then the generating function $F$ of $a(k)$ is equal to
\formula[eq:gf]{
 F(z) & = \sum_{k = -\infty}^\infty a(k) z^k = \exp\biggl(b^+ z + \frac{b^-}{z} + c + \int_{-\infty}^\infty \biggl(\frac{1}{s - z} - \frac{s}{s^2 + 1}\biggr) \ph(s) ds \biggr)
}
when $|z| = 1$, $z \ne 1$, where $b^+, b^- \in [0, \infty)$, $c \in \R$ and $\ph$ is a Borel function on $\R$ such that
\begin{enumerate}[label={\textup{(\roman*)}},leftmargin=2.5em]
\item\label{thm:gf:1} $\ph$ is stepwise decreasing on $(-\infty, 0)$;
\item\label{thm:gf:2} $\ph$ is increasing-after-rounding on $(0, \infty)$;
\item\label{thm:gf:3} $\ph \le 0$ on $(0, 1)$ and $\ph \ge 0$ on $(1, \infty)$;
\item\label{thm:gf:4} $\int_{-\infty}^\infty |\ph(s)| / (s^2 + 1) ds < \infty$;
\item\label{thm:gf:5} $(e^{i t} - 1) F(e^{i t})$ converges to $0$ as $t \to 0$.
\end{enumerate}
Conversely, if $b^+, b^-, c, \ph$ and the function $F$ defined by~\eqref{eq:gf} satisfy the above conditions, then $F$ is the generating function of the convolution $a(k)$ of a summable two-sided Pólya frequency sequence $b(k)$ and an $\amcm$ sequence $c(k)$ which converges to zero as $k \to \pm \infty$.
\end{theorem}

\begin{proof}
By Lemma~\ref{lem:amcm}, the generating function $H$ of an $\amcm$ sequence $c(k)$ which converges to zero as $k \to \pm \infty$ has the exponential representation
\formula{
 H(z) & = \exp\biggl(c_2 + \int_0^\infty \biggl(\frac{1}{s - z} - \frac{s}{s^2 + 1}\biggr) \ph_2(s) ds\biggr)
}
for a constant $c_2 \in \R$ and a Borel function $\ph_2$ on $\R$ which is equal to zero on $(-\infty, 0)$, takes values in $[-1, 0]$ on $(0, 1)$, and takes values in $[0, 1]$ on $(1, \infty)$. Similarly, by Lemma~\ref{lem:pf}, the generating function $G$ of a summable two-sided Pólya frequency sequence $b(k)$ is given by
\formula{
 G(z) & = \exp\biggl(b^+ z + \frac{b^-}{z} + c_1 + \int_{-\infty}^\infty \biggl(\frac{1}{s - z} - \frac{s}{s^2 + 1}\biggr) \ph_1(s) \, ds\biggr)
}
for some constants $b^+, b^- \ge 0$ and $c_1 \in \R$, and a function $\ph_1$ on $\R$ which is stepwise decreasing on $(-\infty, 0)$, is stepwise increasing on $(0, \infty)$, and satisfies $\ph(s) = 0$ in a neighbourhood of $1$. The generating function of the convolution $a(k) = (b * c)(k)$ is a function $F$ satisfying $F(z) = G(z) H(z)$ when $|z| = 1$ and $z \ne 1$. We conclude that $F$ is given by~\eqref{eq:gf}, with $c = c_1 + c_2 \in \R$ and $\ph(s) = \ph_1(s) + \ph_2(s)$. It remains to note that $\ph = \ph_1$ is stepwise decreasing on $(-\infty, 0)$, $\ph = \ph_1 + \ph_2$ is increasing-after-rounding on $(0, \infty)$, and $\ph$ is clearly nonpositive on $(0, 1)$ and nonnegative on $(1, \infty)$. Additionally, both $\ph_1(s) / (s^2 + 1)$ and $\ph_2(s) / (s^2 + 1)$ are integrable (the former by Lemma~\ref{lem:pf}, the latter because $\ph_2$ is bounded), and therefore $|\ph(s)| / (s^2 + 1)$ is integrable over $\R$. Finally, $G$ is continuous on the unit circle in the complex plane and $(e^{i t} - 1) H(e^{i t})$ converges to $0$ as $t \to 0$, so clearly the limit of $(e^{i t} - 1) F(e^{i t})$ as $t \to 0$ is zero.

In order to prove the converse part of the theorem, we first show that every function $\ph$ with the properties listed in the statement of the theorem can be written as a sum $\ph = \ph_1 + \ph_2$, where $\ph_1$ and $\ph_2$ have the properties discussed in the proof of the direct part of the theorem.

For $s < 0$, we simply define $\ph_1(s) = \ph(s)$ and $\ph_2(s) = 0$. For $s > 0$ the definition is slightly more complicated. By definition, there is a stepwise increasing function $\tilde{\ph}$ on $(0, \infty)$ such that $\tilde{\ph}(s) \le \ph(s) \le \tilde{\ph}(s) + 1$ for $s > 0$. Since $\ph(s) \le 0$ for $s \in (0, 1)$ and $\ph(s) \ge 0$ for $s \in (1, \infty)$, with no loss of generality we may assume that $\tilde{\ph}(s) \le -1$ for $s \in (0, 1)$ and $\tilde{\ph}(s) \ge 0$ for $s \in [1, \infty)$ (otherwise we replace $\tilde{\ph}$ by $\min\{\tilde{\ph}, -1\}$ on $(0, 1)$ and by $\max\{\tilde{\ph}, 0\}$ on $[1, \infty)$). We define $\ph_1(s) = \tilde{\ph}(s) + \ind_{(0, 1)}(s)$ and $\ph_2(s) = \ph(s) - \ph_1(s)$. Clearly, $\ph_1$ is stepwise increasing on $(0, \infty)$, $\ph_2(s) \in [-1, 0]$ for $s \in (0, 1)$, and $\ph_2(s) \in [0, 1]$ for $s \in (1, \infty)$. It remains to show that $\ph_1(s) = 0$ in a neighbourhood of $1$. Indeed: if $\ph_1(s) \ge 1$ in some right neighbourhood of $1$, then $\ph(s) \ge 1$ for $s \in (1, \infty)$, which would contradict Lemma~\ref{lem:lim}. Similarly, if $\ph_1(s) \le -1$ in some left neighbourhood of $1$, then $\ph(s) \le -1$ for $s \in (0, 1)$, and again we would arrive at a contradiction with Lemma~\ref{lem:lim}.

We return to the proof of the converse part of the theorem. Suppose that $F$ is given by~\eqref{eq:gf}, define $\ph_1$ and $\ph_2$ as described above, and let $G$ and $H$ be defined as in the proof of the direct part (with, say, $c_1 = 0$ and $c_2 = c$). We observe that since $\ph_1(s) = 0$ in a neighbourhood of $1$, $G$ is continuous on the unit circle in the complex plane and $G(1) \ne 0$. Hence, $(e^{i t} - 1) H(e^{i t}) = (e^{i t} - 1) F(e^{i t}) / G(e^{i t})$ converges to zero as $t \to 0$. Therefore, by Lemma~\ref{lem:amcm}, $H$ is the generating function of an $\amcm$ sequence $c(k)$ which converges to zero as $k \to \pm\infty$. Furthermore, since $\ph_2$ is bounded, $\ph_1(s) / (s^2 + 1)$ is integrable over $\R$, and so Lemma~\ref{lem:pf} implies that $G$ is the generating function of a summable Pólya frequency sequence $b(k)$. The convolution $a(k)$ of $b(k)$ and $c(k)$ has therefore generating function equal to $F(z)$ when $|z| = 1$, $z \ne 1$, and the proof is complete.
\end{proof}


\subsection{Generating functions of bell-shaped sequences}
\label{sec:hard}

We now prove the difficult part of Theorem~\ref{thm:main}: implication \ref{thm:main:a}~$\Rightarrow$~\ref{thm:main:c}. Our argument follows the idea of the proof in the one-sided case in~\cite{kw}, but the the generating function $F$ is only defined on the unit circle, and so a different transform and a nonstandard inversion formula need to be employed. Note that despite these additional difficulties, our new approach also brings some simplifications: we no longer need Step~8 from the proof of Theorem~3.1 in~\cite{kw}.

\begin{theorem}
\label{thm:bell}
Suppose that $a(k)$ is a two-sided bell-shaped sequence. Then the generating function $F$ of $a(k)$ is given by~\eqref{eq:gf}, with $b^+, b^-, c, \ph$ satisfying conditions~\ref{thm:gf:1} through~\ref{thm:gf:5} listed in Theorem~\ref{thm:gf}.
\end{theorem}

\begin{proof}
The argument is divided into nine steps.

\emph{Step 1.} The generating function of $a(k)$ is given by
\formula{
 F(z) & = \sum_{k = -\infty}^\infty a(k) z^k
}
when $|z| = 1$, $z \ne 1$.

If $a(k)$ is summable, then $F$ is continuous and bounded on the unit circle. In this case, when $\im x > 0$, we define
\formula{
 G(x) & = \int_0^\infty e^{i t x} F(e^{i t}) dt .
}
The inversion formula given in Proposition~\ref{prop:post} reads
\formula[eq:step1]{
 F(e^{i t}) & = \lim_{n \to \infty} \frac{(-1)^n (x_n)\rf{n + 1}}{i n!} \, \Delta^n G(x_n) ,
}
where $t \in (0, \pi)$ and
\formula[eq:xn]{
 x_n & = \frac{n}{2} \biggl(i \cot \frac{t}{2} - 1\biggr) .
}
This will be the starting point for our reasoning.

In the general case, $F$ may fail to be integrable over the unit circle in the complex plane, and we need the following minor modification. Since $a(k)$ converges to zero and it is eventually monotone as $k \to \infty$ and as $k \to -\infty$, the sequence $\Delta a(k)$ is summable, and its generating function is equal to $(z - 1) F(z)$ when $|z| = 1$, $z \ne 1$. It follows that $(z - 1) F(z)$ extends to a continuous functions on the unit circle in the complex plane, and hence we may apply the second part of Proposition~\ref{prop:post} to get the same conclusion: \eqref{eq:step1} holds with $x_n$ defined in~\eqref{eq:xn}.

Note that formula~\eqref{eq:step1} is analogous to the result of Step~1 in the proof of Theorem~3.1, but the definition of $G$ is essentially different there, and appropriately defined integers $j_n$ are used there instead of the complex numbers $x_n$ defined above.

\emph{Step 2.} Suppose that $\im x > 0$. If $a(k)$ is integrable, then, using the definition of $F$ and Fubini's theorem, we obtain
\formula{
 G(x) & = \int_0^\infty e^{i t x} \biggl(\sum_{k = -\infty}^\infty a(k) e^{i k t}\biggr) dt \\
 & = \sum_{k = -\infty}^\infty a(k) \biggl(\int_0^\infty e^{i t (x + k)}\biggr) dt \\
 & = \sum_{k = -\infty}^\infty \frac{i a(k)}{x + k} \, .
}
Evaluating the $n$th iterated difference of both sides with respect to $x$, we find that
\formula[eq:step2:aux]{
 \Delta^n G(x) & = \sum_{k = -\infty}^\infty i a(k) \Delta_x^n \frac{1}{x + k} = \sum_{k = -\infty}^\infty i a(k) \Delta_k^n \frac{1}{x + k} \, ,
}
where $\Delta_x$ and $\Delta_k$ denote the forward difference operators with respect to variables $x$ and $k$, respectively. In the general case, by the definition of $\Delta^n G(x)$ (see Proposition~\ref{prop:post}) and Fubini's theorem, for $n \ge 1$ we have
\formula{
 \Delta^n G(x) & = \int_0^\infty e^{i t x} (e^{i t} - 1)^n \biggl(\sum_{k = -\infty}^\infty a(k) e^{i k t}\biggr) dt \\
 & = \sum_{k = -\infty}^\infty a(k) \biggl(\int_0^\infty (e^{i t} - 1)^n e^{i t (x + k)}\biggr) dt \\
 & = \sum_{k = -\infty}^\infty a(k) \biggl(\int_0^\infty \Delta_k^n e^{i t (x + k)}\biggr) dt \\
 & = \sum_{k = -\infty}^\infty a(k) \Delta_k^n \biggl(\int_0^\infty e^{i t (x + k)}\biggr) dt \\
 & = \sum_{k = -\infty}^\infty a(k) \Delta_k^n \frac{i}{x + k} \, ,
}
and we come to the same conclusion~\eqref{eq:step2:aux}.

Let $P$ be a polynomial of degree at most $n$. As in Step~2 in the proof of Theorem~3.1 in~\cite{kw}, we observe that $(P(k) - P(-x)) / (x + k)$ is a polynomial in $k$ of degree at most $n - 1$, and hence
\formula{
 \Delta_k^n \frac{P(k)}{x + k} & = P(-x) \Delta_k^n \frac{1}{x + k} \, .
}
Therefore,
\formula{
 P(-x) \Delta^n G(x) & = \sum_{k = -\infty}^\infty i a(k) \Delta_k^n \frac{P(k)}{x + k} \, .
}
Applying summation by parts $n$ times to the right-hand side, we conclude that
\formula[eq:step2]{
 P(-x) \Delta^n G(x) & = (-1)^n i \sum_{k = -\infty}^\infty \frac{P(k)}{x + k} \, \Delta^n a(k - n) .
}
Note that in each application of summation by parts, the boundary terms vanish: for every $j = 0, 1, \ldots, n - 1$, the sequence $\Delta_k^{n - 1 - j}(P(x) / (x + k))$ is bounded by a constant times $(1 + |k|)^j$, and the sequence $(1 + |k|)^j \Delta^j a(k - j)$ converges to zero as $k \to \pm \infty$ by Lemma~2.3 in~\cite{kw}.

\emph{Step 3.} Combining \eqref{eq:step1} and \eqref{eq:step2}, we obtain
\formula[eq:step3:1]{
 F(e^{i t}) & = \lim_{n \to \infty} \frac{(x_n)\rf{n + 1}}{n! P_n(-x_n)} \sum_{k = -\infty}^\infty \frac{P_n(k) \Delta^n a(k - n)}{x_n + k} \, ,
}
where $t \in (0, \pi)$, $x_n$ is given by~\eqref{eq:xn}, and $P_n$ is an arbitrary sequence of polynomials of degree at most $n$. We choose these polynomials in such a way that
\formula{
 P_n(k) \Delta^n a(k - n) & \ge 0
}
for every $k \in \Z$: we set
\formula{
 P_n(k) & = \prod_{j = 0}^{n - 1} (\alpha_{n, j} - k) ,
}
where $\alpha_{n, 0}, \alpha_{n, 1}, \ldots, \alpha_{n, n - 1}$ denote the locations of sign changes of the sequence $\Delta^n a(k - n)$. To be specific, we let $\alpha_{n,-1} = -\infty$, and we define inductively
\formula{
 \alpha_{n, j} & = \min \{k > \alpha_{n, j - 1} : (-1)^j \Delta^n a(k - n) > 0 \}
}
for $j = 0, 1, 2, \dots, n - 1$. With this choice of $P_n$, formula~\eqref{eq:step3:1} can be rewritten as
\formula[eq:step3:2]{
 F(e^{i t}) & = \lim_{n \to \infty} \biggl(\frac{x_n + n}{n!} \prod_{j = 0}^{n - 1} \frac{x_n + j}{x_n + \alpha_{n,j}}\biggr) \sum_{k = -\infty}^\infty \frac{P_n(k) \Delta^n a(k - n)}{x_n + k} \, .
}
This part is very similar to Step~3 of the proof of Theorem~3.1 in~\cite{kw}.

\emph{Step 4.} Recall that in formula~\eqref{eq:step3:2}, $t \in (0, \pi)$ and $x_n$ is given by~\eqref{eq:xn}. On the right-hand side of~\eqref{eq:step3:2} only $x_n$ depends on $t$, namely,
\formula{
 x_n & = \frac{n}{2} \biggl(i \cot \frac{t}{2} - 1\biggr) = \frac{n}{2} \biggl(-\frac{e^{i t / 2} + e^{-i t / 2}}{e^{i t / 2} - e^{-i t / 2}} - 1\biggr) = \frac{n e^{i t}}{1 - e^{i t}} \, .
}
When $x \in \C \setminus \Z$, we denote by $G_n(x)$ the expression under the limit in~\eqref{eq:step3:2}, with $x_n$ replaced by $x$:
\formula[eq:step4]{
 G_n(x) & = \biggl(\frac{x + n}{n!} \prod_{j = 0}^{n - 1} \frac{x + j}{x + \alpha_{n,j}}\biggr) \sum_{k = -\infty}^\infty \frac{P_n(k) \Delta^n a(k - n)}{x + k} \, .
}
Thus, $G_n(x_n)$ converges to $F(e^{i t})$ as $n \to \infty$. This part is analogous to Step~4 of the proof of Theorem~3.1 in~\cite{kw}. In the next two steps, we denote the two factors in the definition~\eqref{eq:step4} of $G_n$ by $G^\flat_n$ and $G^\sharp_n$, and we study each of them separately.

\emph{Step 5.} For $x \in \C \setminus \Z$ we denote
\formula[eq:step5:1]{
 G^\sharp_n(x) & = \sum_{k = -\infty}^\infty \frac{P_n(k) \Delta^n a(k - n)}{x + k} \, .
}
Since $P_n(k) \Delta^n a(k - n) \ge 0$, we have $\im G^\sharp_n(x) \le 0$ when $\im x > 0$. Since $G^\sharp_n$ is not identically zero, it follows that $1 / G^\sharp_n$ is a Pick function. Note that $G^\sharp_n$ is real-valued and decreasing on each interval $(k, k + 1)$, $k \in \Z$. Hence, the exponential representation of $1 / G^\sharp_n$ is given by
\formula[eq:step5:2]{
 \frac{1}{G^\sharp_n(x)} & = \exp\biggl(c^\sharp_n + \int_{-\infty}^\infty \biggl(\frac{1}{s - x} - \frac{s}{s^2 + 1}\biggr) \psi^\sharp_n(s) ds\biggr)
}
when $\im x > 0$, where $c^\sharp_n \in \R$ and
\formula{
 \psi^\sharp_n(s) & = \lim_{t \to 0^+} \frac{1}{\pi} \Arg \frac{1}{G^\sharp_n(s + i t)} = \begin{cases} 1 & \text{if $G^\sharp_n(s) < 0$,} \\ 0 & \text{if $G^\sharp_n(s) > 0$.} \end{cases}
}
Furthermore, for each $j \in \Z$ there is a number $\beta_{n, j} \in [-j - 1, -j]$ such that
\formula{
 \{s \in (-j - 1, -j) : G^\sharp_n(s) < 0\} & = (-\beta_{n, j}, -j) ,
}
and hence
\formula[eq:step5:3]{
 \psi^\sharp_n(s) & = \sum_{j = -\infty}^\infty \ind_{[-\beta_{n, j}, -j)}(s)
}
for almost every $s \in \R$. This part is essentially the same as Step~5 in the proof of Theorem~3.1 in~\cite{kw}.

\emph{Step 6.} The other factor in the definition~\eqref{eq:step4} of $G_n$ reads
\formula{
 G^\flat_n(x) & = \frac{x + n}{n!} \prod_{j = 0}^{n - 1} \frac{x + j}{x + \alpha_{n, j}}
}
for $x \in \C \setminus \Z$. Its exponential representation in the upper half-plane $\im x > 0$ follows from the elementary identities
\formula{
 \log(x + n) & = \frac{1}{2} \, \log(n^2 + 1) + \int_{-\infty}^{-n} \biggl(\frac{1}{s - x} - \frac{s}{s^2 + 1}\biggr) ds , \displaybreak[0] \\
 \log \frac{x + j}{x + \alpha_{n,j}} & = \frac{1}{2} \log \frac{j^2 + 1}{\alpha_{n, j}^2 + 1} + \int_{-\alpha_{n,j}}^{-j} \biggl(\frac{1}{s - x} - \frac{s}{s^2 + 1}\biggr) ds .
}
By definition, we have
\formula[eq:step6:1]{
 G^\flat_n(x) = \exp\biggl(c^\flat_n + \int_{-\infty}^\infty \biggl(\frac{1}{s - x} - \frac{s}{s^2 + 1}\biggr) \psi^\flat_n(s) \, ds  \biggr),
}
where $c^\flat_n \in \R$ and
\formula[eq:step6:2]{
 \psi^\flat_n(s) & = \sum_{j = 0}^n \ind_{(-\infty, -j)}(s) - \sum_{j = 0}^{n - 1} \ind_{(-\infty, -\alpha_{n,j})}(s) .
}
This part is identical to Step~6 of the proof of Theorem~3.1 in~\cite{kw}.

\emph{Step 7.} Combining \eqref{eq:step5:2} and \eqref{eq:step6:1} together, we obtain the exponential representation of the function $G_n$ defined in~\eqref{eq:step4}, namely
\formula[eq:step7:0]{
 G_n(x) & = \exp\biggl(a_n + \int_{-\infty}^\infty \biggl(\frac{1}{s - x} - \frac{s}{s^2 + 1}\biggr) \psi_n(s) ds \biggr) ,
}
where $a_n = a^\flat_n - a^\sharp_n \in \R$ and $\psi_n = \psi^\flat_n - \psi^\sharp_n$. In this step we prove that $\psi_n$ is stepwise decreasing on $(-n, 0)$, increasing-after-rounding and nonnegative on $(-\infty, -n)$, and increasing-after-rounding and nonpositive on $(0, \infty)$. This is the counterpart of Step~7 of the proof of Theorem~3.1 in~\cite{kw}, but the analysis requires more care in the present case.

To simplify the notation, in this step only we fix $n$, and we write $\alpha_{n, j} = \alpha_j$ and $\beta_{n, j} = \beta_j$. By~\eqref{eq:step5:2} and~\eqref{eq:step6:2}, we have
\formula{
 \psi_n(s) & = \sum_{j = 0}^n \ind_{(-\infty, -j)}(s) - \sum_{j = 0}^{n - 1} \ind_{(-\infty, -\alpha_j)}(s) - \sum_{j = -\infty}^\infty \ind_{[-\beta_j, -j)}(s) .
}
Recall that $-j - 1 \le -\beta_j \le -j$. Elementary manipulations lead to
\formula{
 \psi_n & = \sum_{j = 0}^{n - 1} \bigl(\ind_{(-\infty, -j)} - \ind_{(-\infty, -\alpha_j)}\bigr) + \biggl(\sum_{j = n}^\infty \ind_{[-j - 1, -j)} - \sum_{j = -\infty}^\infty \ind_{[-\beta_j, -j)}\biggr) \\
 & = \sum_{j = 0}^{n - 1} \bigl(\ind_{(-\infty, -j)} - \ind_{(-\infty, -\alpha_j)} - \ind_{[-\beta_j, -j)}\bigr) \\
 & \hspace*{3em} - \sum_{j = -\infty}^{-1} \ind_{[-\beta_j, -j)} + \sum_{j = n}^\infty \bigl(\ind_{[-j - 1, -j)} - \ind_{[-\beta_j, -j)}\bigr) \\
 & = \sum_{j = 0}^{n - 1} \bigl(\ind_{(-\infty, -\beta_j)} - \ind_{(-\infty, -\alpha_j)}\bigr) - \sum_{j = -\infty}^{-1} \ind_{[-\beta_j, -j)} + \sum_{j = n}^\infty \ind_{[-j - 1, -\beta_j)} .
}
The last two terms on the right-hand side define a function which takes values in $\{0, 1\}$ on $(-\infty, -n)$, which is equal to zero on $(-n, 0)$, and which takes values in $\{-1, 0\}$ on $(0, \infty)$. Let us now inspect the first term, that is,
\formula[eq:step7:2]{
 \tilde{\psi}_n(s) = \sum_{j = 0}^{n - 1} \bigl(\ind_{(-\infty, - \beta_j)}(s) - \ind_{(-\infty, -\alpha_j)}(s)\bigr) .
}
This function only takes integer values, it has $n$ upward jumps at $-\alpha_j$, and it has $n$ downward jumps at $-\beta_j$, where $j = 0, 1, \ldots, n - 1$. In particular, $\tilde{\psi}_n(s) = 0$ outside a finite interval. Since $-\beta_j \in [-n, 0]$ for every $j = 0, 1, \ldots, n - 1$, the function $\tilde{\psi}_n$ is stepwise increasing on $(-\infty, -n)$ and on $(0, \infty)$.

It follows that on $(-\infty, -n)$, the function $\psi_n$ is a sum of a stepwise increasing nonnegative function $\tilde{\psi}_n$ and a function that takes values in $\{0, 1\}$. Thus, $\psi_n$ is increasing-after-rounding and nonnegative on $(-\infty, n)$. Similarly, $\psi_n$ is increasing-after-rounding and nonpositive on $(0, \infty)$. It remains to show that $\psi_n$ is stepwise decreasing on $(-n, 0)$.

We already know that $\psi_n$ takes integer values on $(-n, 0)$ and it has downward jumps at $-\beta_j \in [-j - 1, -j]$ for $j = 0, 1, \ldots, n - 1$. It may have an upward jump at $-\alpha_j$ with $j = 0, 1, \ldots, n - 1$, as long as this number belongs to $(-n, 0)$. Suppose that $-\alpha_j \in (-n, 0)$, that is, $\alpha_j \in \{1, 2, \ldots, n - 1\}$, and write $m = \alpha_j$. By definition, $P_n(m) = 0$, and so the function $G^\sharp_n$ defined in~\eqref{eq:step5:1} does not have a pole at $y = -m$. Therefore, $G^\sharp_n$ is decreasing on $(-m - 1, -m + 1)$.

If $G^\sharp_n(-m) < 0$, then $G^\sharp_n(y) < 0$ for $y \in (-m, -m + 1)$, and so, by definition, $\beta_{m - 1} = m$. In this case the upward jump of $\psi_n$ at $-\alpha_j = -m$ is cancelled by a downward jump at $-\beta_{m - 1} = -m$.

If $G^\sharp_n(-m) \ge 0$, then $G^\sharp_n(y) \ge 0$ for $y \in (-m - 1, -m)$, and hence $\beta_m = m$. In this case the upward jump of $\psi_n$ at $-\alpha_j = -m$ is cancelled by a downward jump at $-\beta_m = -m$.

The above argument shows that every downward jump of $\psi_n$ on $(-n, 0)$ is cancelled by an appropriate upward jump, and so $\psi_n$ is indeed stepwise decreasing on $(-n, 0)$, as desired.

\emph{Step 8.} Recall that $x_n = n e^{i t} / (1 - e^{i t})$, and our ultimate goal is to express $F(e^{i t})$ in terms of $e^{i t}$. For this reason we substitute $x = n z / (1 - z)$ in the exponential representation~\eqref{eq:step7:0} of the function $G_n$. In other words, when $\im z > 0$ we define
\formula{
 F_n(z) & = G_n\biggl(\frac{n z}{1 - z}\biggr) = \exp\biggl(c_n + \int_{-\infty}^\infty \biggl(\frac{1 - z}{s (1 - z) - n z} - \frac{s}{s^2 + 1}\biggr) \psi_n(s) ds\biggr) ,
}
so that $F_n(e^{i t}) = G_n(x_n)$ converges to $F(e^{i t})$ for every $t \in (0, \pi)$. As in Step~9 of the proof of Theorem~3.1 in~\cite{kw}, we substitute $s = n r / (1 - r)$ in the integral on the right-hand side. In order to do so, we denote
\formula{
 \ph_n(r) & = \psi_n\biggl(\frac{n r}{1 - r}\biggr) .
}
It is straightforward to see that $\ph_n$ is stepwise decreasing on $(-\infty, 0)$, increasing-after-rounding and nonpositive on $(0, 1)$, and increasing-after-rounding and nonnegative on $(1, \infty)$. Substitution leads to
\formula[eq:step8]{
 F_n(z) & = \exp\biggl(d_n + \int_{-\infty}^\infty \biggl(\frac{1}{r - z} - \frac{r}{r^2 + 1}\biggr) \ph_n(r) dr\biggr)
}
when $\im z > 0$ for an appropriate $d_n \in \R$; we omit the easy details, which are exactly the same as in the corresponding part of~\cite{kw}.

\emph{Step 9.} We know that $F_n(z)$ is given by~\eqref{eq:step8} when $\im z > 0$, and that $F_n(e^{i t})$ converges to $F(e^{i t})$ as $n \to \infty$ for every $t \in (0, \pi)$. The desired representation~\eqref{eq:gf} of the limiting function $F$ follows now from Lemma~\ref{lem:compact}. Clearly, $\ph$ satisfies conditions~\ref{thm:gf:1} through~\ref{thm:gf:4} of Theorem~\ref{thm:gf}. In order to prove the remaining condition~\ref{thm:gf:5}, observe, as in Remark~\ref{rem:c5}, that the sequence $\Delta a(k)$ is summable and it sums up to $0$, and hence its generating function is continuous on the unit circle in the complex plane, and it takes value $0$ at $z = 1$. The generating function of $\Delta a(k)$ is, however, equal to $(z - 1) F(z)$ when $|z| = 1$, $z \ne 1$, and hence the limit of $(e^{i t} - 1) F(e^{i t})$ as $t \to 0$ is equal to zero.
\end{proof}

%
%

\section{Proof of the sampling theorem}
\label{sec:sample}

In this short section we prove Theorem~\ref{thm:sample}.

\begin{proof}[Proof of Theorem~\ref{thm:sample}]
By Theorems~1.1 and~1.3 in~\cite{ks}, $f$ is bell-shaped if and only if $f = g * h$, where $g$ is an $\amcm$ function and $h$ is a Pólya frequency function (we refer to~\cite{kwasnicki,ks} for a detailed discussion of these classes of functions). Let $\Delta f(x) = f(x + 1) - f(x)$ be the forward difference of the function $f$, and for $n = 0, 1, 2, \ldots$ let $\Delta^n f(x)$ be the $n$th iterated forward difference.

Clearly, $\Delta^n f = (\Delta^n g) * h$, and $\Delta^n g(x)$ converges to zero as $x \to \pm \infty$. We claim that $\Delta^n g$ is a monotone function on each of the intervals
\formula{
 & (-\infty, -n] , \, [-n, -n + 1] , \, [-n + 1, -n + 2] , \, \ldots \, , [-1, 0] , \, [0, \infty) .
}
Let us postpone the proof of this claim and proceed with the proof of the theorem. By monotonicity, $\Delta^n g$ changes sign at most once in each of the finite intervals $[-n, -n + 1]$, $[-n + 1, -n + 2]$, \ldots, $[-1, 0]$. Furthermore, being monotone on the infinite intervals $(-\infty, -n]$ and $[0, \infty)$, and convergent to zero at $\pm \infty$, $\Delta^n g$ has constant sign on these two intervals. Hence, $\Delta^n g$ changes sign at most $n$ times.

Since $h$ is a Pólya frequency function, it has the variation diminishing property: the convolution with $h$ does not increase the number of sign changes. Hence, also $\Delta^n f = (\Delta^n g) * h$ changes sign at most $n$ times. (Strictly speaking, the variation diminishing property applies to bounded functions, so here we first approximate $\Delta^n f$ by a sequence of bounded functions, then apply the variation diminishing property to these approximations, and finally pass to the limit. We omit the details.)

The sequence $\Delta^n f(k)$ changes sign at most as many times as the function $\Delta^n f(x)$, and we have already proved that $\Delta^n f(x)$ changes sign at most $n$ times. It follows that for every $n = 0, 1, 2, \ldots$\,, $\Delta^n f(k)$ changes sign at most $n$ times. By the discrete analogue of Rolle's theorem (see Section~\ref{sec:gen}) and induction, $\Delta^n f(k)$ changes sign at least $n$ times, and thus $f(k)$ is a bell-shaped sequence.

It remains to prove our claim. We tacitly assume that $g$ is upper semi-continuous at $0$, with possibly $g(0) = \infty$ (also when $g$ contains an atom at $0$; we refer to~\cite{kwasnicki,ks} for a detailed discussion). By Bernstein's theorem, for $x > 0$ we have
\formula{
 g(x) & = \int_{(0, \infty)} e^{-s x} \mu_+(ds) , & g(-x) & = \int_{(0, \infty)} e^{-s x} \mu_-(ds) ,
}
where $\mu_+$ and $\mu_-$ are nonnegative measures on $(0, \infty)$ such that the above integrals converge for every $x > 0$. We can extend these equalities to $x \ge 0$ by adding an appropriate atom of $\mu_+$ and $\mu_-$ at $\infty$ and extending the range of integration to $(0, \infty]$ (with $e^{-\infty} = 0$ and $e^{-\infty \cdot 0} = 1$). By simple induction,
\formula{
 (-1)^n \Delta^n g(x) & = \int_{(0, \infty]} e^{-s x} (1 - e^{-s})^n \mu_+(ds) , \\
 \Delta^n g(-n - x) & = \int_{(0, \infty]} e^{-s x} (1 - e^{-s})^n \mu_-(ds)
}
for $x \ge 0$. In particular, $\Delta^n g$ is absolutely monotone on $(-\infty, -n]$, and $(-1)^n \Delta^n g$ is completely monotone on $[0, \infty)$. This proves our claim for the two infinite intervals.

Finite intervals are now handled by induction: we prove that $(-1)^j \Delta^n g$ is increasing on $[-n + j - 1, -n + j]$ for $j = 1, 2, \ldots, n$. The base case $n = 0$ is an empty statement. Assume the inductive hypothesis for a given $n = 0, 1, 2, \ldots$\, By what we have proved for the two infinite intervals, $(-1)^j \Delta^n g$ is increasing on $[-n + j - 1, -n + j]$ also for $j = 0$ and $j = n + 1$. It follows that for $j = 0, 1, \ldots, n$ and $x \in [-n + j - 1, -n + j]$,
\formula{
 (-1)^{j + 1} \Delta^{n + 1} g(x) = (-1)^{j + 1} \Delta^n g(x + 1) + (-1)^j \Delta^n g(x)
}
is an increasing function. This completes the inductive step, and our claim follows.
\end{proof}

%
%

\appendix

\section{Auxiliary `approximation to identity' lemma}
\label{app:a}

The following result is used in the proof of Proposition~\ref{prop:post}

\begin{lemma}
\label{lem:approximation}
Suppose that $\ph$ is a continuous integrable real-valued function on $[0, \infty)$ such that $\ph(s)$ converges to $0$ as $s \to \infty$, $|\ph|$ attains a strict global maximum at a point $t \in (0, \infty)$, and $\ph(t) > 0$. Denote
\formula{
 M_n & = \int_0^\infty (\ph(s))^n ds .
}
If $F$ is an integrable function on the unit circle in the complex plane and $F$ is continuous at $e^{i t}$, then
\formula{
 \lim_{n \to \infty} \frac{1}{M_n} \int_0^\infty (\ph(s))^n F(e^{i s}) ds & = F(e^{i t}) .
}
Instead of integrability of $F$ it is sufficient to assume that $(\ph(t))^n F(e^{i t})$ is integrable over $(0, \infty)$ for $n$ large enough.
\end{lemma}

\begin{proof}
Our goal is to prove that
\formula{
 \lim_{n \to \infty} \frac{1}{M_n} \int_0^\infty (\ph(s))^n (F(e^{i s}) - F(e^{i t})) ds & = 0 .
}
Fix $\eps > 0$, and choose $\delta > 0$ small enough, so that $|F(e^{i s}) - F(e^{i t})| < \eps$ and $\ph(s) > 0$ whenever $|s - t| < \delta$. By assumptions, there is $\thet > 0$ such that $|\ph(s)| \le (1 - 2 \thet) \ph(t)$ when $s \ge 0$ and $|s - t| \ge \delta$. Furthermore, we can find $\eta > 0$ such that $\ph(s) > (1 - \thet) \ph(t)$ when $|s - t| < \eta$. On one hand, we have
\formula*[eq:approximation:1]{
 \hspace*{3em} & \hspace*{-3em} \biggl|\int_{[0, \infty) \setminus (t - \delta, t + \delta)} (\ph(s))^n ds\biggr| \\
 & \le \int_{[0, \infty) \setminus (t - \delta, t + \delta)} |\ph(s)|^n ds \\
 & \le ((1 - 2 \thet) \ph(t))^{n - 1} \int_{[0, \infty) \setminus (t - \delta, t + \delta)} |\ph(s)| ds
}
and, in the same vein,
\formula*[eq:approximation:2]{
 \hspace*{3em} & \hspace*{-3em} \biggl|\int_{[0, \infty) \setminus (t - \delta, t + \delta)} (\ph(s))^n (F(e^{i s}) - F(e^{i t})) ds\biggr| \\
 & \le \int_{[0, \infty) \setminus (t - \delta, t + \delta)} |\ph(s)|^n |F(e^{i s}) - F(e^{i t})| ds \\
 & \le ((1 - 2 \thet) \ph(t))^{n - 1} \int_{[0, \infty) \setminus (t - \delta, t + \delta)} |\ph(s)| |F(e^{i s}) - F(e^{i t})| ds .
}
On the other one, we have
\formula[eq:approximation:3]{
 M_n & \ge \int_{(t - \delta, t + \delta)} (\ph(s))^n ds \ge \int_{(t - \eta, t + \eta)} (\ph(s))^n ds \ge 2 \eta ((1 - \thet) \ph(t))^n ,
}
and
\formula[eq:approximation:4]{
 \biggl|\int_{(t - \delta, t + \delta)} (\ph(s))^n (F(e^{i s}) - F(e^{i t})) ds \biggr| & \le \eps \int_{(t - \delta, t + \delta)} (\ph(s))^n ds \le \eps M_n .
}
Since $(1 - \thet) \ph(t) > (1 - 2 \thet) \ph(t)$, the right-hand sides of~\eqref{eq:approximation:1} and~\eqref{eq:approximation:2} are negligible compared to the right-hand side of~\eqref{eq:approximation:3} as $n \to \infty$. Therefore,
\formula*{
 \lim_{n \to \infty} \frac{1}{M_n} \int_{[0, \infty) \setminus (t - \delta, t + \delta)} (\ph(s))^n ds = 0 , \\
 \lim_{n \to \infty} \frac{1}{M_n} \int_{(t - \delta, t + \delta)} (\ph(s))^n ds = 1 ,
}
and
\formula{
 \lim_{n \to \infty} \frac{1}{M_n} \biggl|\int_{[0, \infty) \setminus (t - \delta, t + \delta)} (\ph(s))^n (F(e^{i s}) - F(e^{i t})) ds\biggr| & = 0 .
}
By~\eqref{eq:approximation:4}, we also have
\formula{
 \limsup_{n \to \infty} \frac{1}{M_n} \biggl|\int_{(t - \delta, t + \delta)} (\ph(s))^n (F(e^{i s}) - F(e^{i t})) ds \biggr| & \le \eps .
}
The last two formulae imply that
\formula{
 \limsup_{n \to \infty} \frac{1}{M_n} \biggl|\int_0^\infty (\ph(s))^n (F(e^{i s}) - F(e^{i t})) ds\biggr| & \le \eps .
}
Since $\eps > 0$ is arbitrary, the desired result follows.

If we relax the integrability assumption on $F$ to integrability of $(\ph(t))^m F(e^{i t})$ over $(0, \infty)$ for some $m$, then we only need to replace the bound~\eqref{eq:approximation:2} by
\formula{
 \hspace*{3em} & \hspace*{-3em} \biggl|\int_{[0, \infty) \setminus (t - \delta, t + \delta)} (\ph(s))^n (F(e^{i s}) - F(e^{i t})) ds\biggr| \\
 & \le ((1 - 2 \thet) \ph(t))^{n - m} \int_{[0, \infty) \setminus (t - \delta, t + \delta)} |\ph(s)|^m |F(e^{i s}) - F(e^{i t})| ds .
}
Otherwise, the proof is exactly the same.
\end{proof}

%
%

\section{Holomorphic character of two functions}
\label{app:b}

Below we give a detailed derivation of the representation formula for two holomorphic functions which appear in Examples~\ref{ex:rw} and~\ref{ex:rrw}.

\subsection{Independent simple random walks}

We begin with the function $F_1$ defined in Example~\ref{ex:rrw}. For simplicity, we drop the index $1$ from the notation, that is, we consider the function
\formula{
 F(z) & = \frac{2 + i (z^{1/2} - z^{-1/2})}{z + 1} \, .
}
Our goal is to prove that $F$ has the form given in~\eqref{eq:main}, so that Theorem~\ref{thm:main} applies. We have already observed that $F$ is a holomorphic function in the upper complex half-plane $\im z > 0$. Observe that in this region
\formula{
 2 + i (z^{1/2} - z^{-1/2}) & = 2 - (-z)^{1/2} - (-z)^{-1/2} = \bigl(i (-z)^{1/4} - i (-z)^{-1/4}\bigr)^2 .
}
Furthermore, $\re (i (-z)^{1/4} - i (-z)^{-1/4}) > 0$, and hence
\formula{
 \Arg\bigl(2 + i (z^{1/2} - z^{-1/2})\bigr) & = 2 \Arg \bigl(i (-z)^{1/4} - i (-z)^{-1/4}\bigr)
}
is well-defined and takes values in $(-\pi, \pi)$. Clearly, $\Arg(z + 1)$ is well-defined and takes values in $(0, \pi)$ when $\im z > 0$. It follows that the argument of $F$ has a continuous version in the upper complex half-plane $\im z > 0$, which we denote by $\pi \Phi$, and
\formula{
 \Phi(z) & = \tfrac{1}{\pi} \Arg\bigl(2 + i (z^{1/2} - z^{-1/2})\bigr) - \tfrac{1}{\pi} \Arg(z + 1) .
}
Furthermore, $\Phi(z) \in [-2, 1]$ for every $z$, and hence $\Phi$ is a bounded harmonic function in the upper complex half-plane. By Poisson's representation formula,
\formula{
 \Phi(z) & = \frac{1}{\pi} \int_{-\infty}^\infty \im \frac{1}{s - z} \, \ph(s) ds ,
}
where $\ph(s) = \lim_{t \to 0^+} \Phi(s + i t)$ is the boundary limit of $\Phi$, defined for almost every $s \in \R$. Recall that $\pi \Phi$ is the imaginary part of the continuous logarithm of $F$. It follows that for some constant $c \in \R$,
\formula{
 F(z) & = \exp\biggl(c + \int_{-\infty}^\infty \biggl(\frac{1}{s - z} - \frac{s}{s^2 + 1}\biggr) \ph(s) ds\biggr) ,
}
that is, $F$ is indeed given by~\eqref{eq:main}. Furthermore, for $s \ne -1$, we have
\formula{
 \lim_{t \to 0^+} \Arg(s + i t + 1) & = \pi \ind_{(-\infty, -1)}(s) .
}
When $s > 0$, clearly
\formula{
 \lim_{t \to 0^+} \Arg\bigl(2 + i ((s + i t)^{1/2} - (s + i t)^{-1/2})\bigr) & = \Arg\bigl(2 + i (s^{1/2} - s^{-1/2})\bigr) \\
 & = \arctan \frac{s^{1/2} - s^{-1/2}}{2} \, .
}
Finally, when $s < 0$ and $s \ne -1$, we have
\formula{
 \hspace*{3em} & \hspace*{-3em} \lim_{t \to 0^+} \Arg\bigl(2 + i ((s + i t)^{1/2} - (s + i t)^{-1/2})\bigr) \\
 & = \lim_{t \to 0^+} \Arg\bigl(2 - (-s)^{1/2} - (-s)^{-1/2} + \tfrac{1}{2} i (-s)^{-1/2} t - \tfrac{1}{2} i (-s)^{-3/2} t + O(t^2)\bigr) \\
 & = \pi \ind_{(-\infty, -1)}(s) - \pi \ind_{(-1, 0)}(s) .
}
It follows that
\formula{
 \pi \ph(s) & = \lim_{t \to 0^+} \pi \Phi(s + i t) \\
 & = \lim_{t \to 0^+} \Arg\bigl(2 + i ((s + i t)^{1/2} - (s + i t)^{-1/2})\bigr) - \lim_{t \to 0^+} \Arg(s + i t + 1)
}
has the desired form, namely,
\formula{
 \ph(s) & = \begin{cases}
  0 & \text{if $s < -1$;} \\
  -1 & \text{if $-1 < s < 0$;} \\
  \tfrac{1}{\pi} \arctan (\tfrac{1}{2} (s^{1/2} - s^{-1/2})) & \text{if $s > 0$,}
 \end{cases}
}
as claimed in Example~\ref{ex:rrw}. We note that in fact we have $\ph(s) \in [-1, 1]$ for almost every $s \in \R$, and thus $\Phi(z) \in (-1, 1)$ for every $z$ in the upper complex half-plane $\im z > 0$, that is, $\Phi$ is in fact the principal branch of the argument of $F(z)$.

\subsection{Two-dimensional simple random walk}

We turn to the properties of the function $F_1$ introduced in Example~\ref{ex:rw}. Again we drop the index $1$ from the notation, that is, we consider the function
\formula{
 F(z) & = 2 - \frac{z + z^{-1}}{2} - \sqrt{1 - \frac{z + z^{-1}}{2}} \sqrt{3 - \frac{z + z^{-1}}{2}} \, .
}
In Example~\ref{ex:rw} we argued that $F$ is a holomorphic function in the upper complex half-plane $\im z > 0$, and in this region $1 - \tfrac{1}{2} (z + z^{-1}) \in \C \setminus (-\infty, 0]$. If we write $w = 1 - \tfrac{1}{2} (z + z^{-1})$, then $w \in \C \setminus (-\infty, 0]$, and
\formula{
 F(z) & = 1 + w - \sqrt{w} \sqrt{w + 2} = \frac{(1 + w)^2 - w (w + 2)}{1 + w + \sqrt{w} \sqrt{w + 2}} = \frac{1}{1 + w + \sqrt{w} \sqrt{w + 2}} \, .
}
If $w > 0$, then $1 + w + \sqrt{w} \sqrt{w + 2} > 0$. Suppose that $\im w > 0$. Then $0 < \Arg(w + 2) < \Arg w < \pi$, and hence $0 < \Arg(\sqrt{w} \sqrt{w + 2}) < \pi$. Therefore, $\im(1 + w + \sqrt{w} \sqrt{w + 2}) > 0$. Similarly, if $\im w < 0$, then $\im(1 + w + \sqrt{w} \sqrt{w + 2}) < 0$. It follows that $F(z) \in \C \setminus (-\infty, 0]$. Consequently, $\Phi(z) = \tfrac{1}{\pi} \Arg F(z)$ is well-defined in the upper complex half-plane with the principal branch of the logarithm, and $\Phi(z) \in [-1, 1]$. As in the previous section, we use Poisson's representation formula for the bounded harmonic function $\Phi$ to find that $F$ is indeed given by~\eqref{eq:main}, with
\formula{
 \pi \ph(s) & = \lim_{t \to 0^+} \Arg F(s + i t) .
}
In order to evaluate the above limit, we observe that if $z = s + i t$ and, as before, $w = 1 - \tfrac{1}{2} (z + z^{-1})$, then
\formula{
 w & = 1 - \tfrac{1}{2} (s + s^{-1}) - \tfrac{1}{2} i t + \tfrac{1}{2} i s^{-2} t + O(t^2)
}
as $t \to 0^+$. Thus, as $t \to 0^+$, $w$ converges to $1 - \tfrac{1}{2} (s + s^{-1})$. Additionally, $\im w$ is negative for small $t > 0$ when $|s| > 1$ and positive for small $t > 0$ when $|s| < 1$.

In terms of variable $w$ introduced above, we have
\formula{
 \lim_{t \to 0^+} \Arg F(s + i t) & = -\lim_{t \to 0^+} \Arg\bigl(1 + w + \sqrt{w} \sqrt{w + 2}\bigr) .
}
When $s < 0$, the limit of $w$ is positive, and therefore $F(s + i t)$ converges to a positive limit as $t \to 0^+$. That is,
\formula{
 \lim_{t \to 0^+} \Arg F(s + i t) & = 0 .
}
When $0 < s < 3 - 2 \sqrt{2}$, then $w$ converges to a number in $(-\infty, -2)$, so that $F(z)$ converges to a negative number. Furthermore, the imaginary part of $w$ is positive for $t$ small enough. Therefore,
\formula{
 \lim_{t \to 0^+} \Arg F(s + i t) & = -\lim_{t \to 0^+} \Arg\bigl(1 + w + \sqrt{w} \sqrt{w + 2}\bigr) = -\pi .
}
Similarly, when $s > 3 + 2 \sqrt{2}$, then again $w$ converges to a number in $(-\infty, -2)$ and $F(z)$ converges to a negative number, but the imaginary part of $w$ is negative for $t$ small enough. It follows that
\formula{
 \lim_{t \to 0^+} \Arg F(s + i t) & = -\lim_{t \to 0^+} \Arg\bigl(1 + w + \sqrt{w} \sqrt{w + 2}\bigr) = \pi .
}
Finally, when $3 - 2 \sqrt{2} < s < 3 + 2 \sqrt{2}$ and $s \ne 1$, then the limit of $w$ lies in $(-2, 0)$, and hence
\formula{
 \lim_{t \to 0^+} \Arg F(s + i t) & = -\lim_{t \to 0^+} \Arg\bigl(1 + w + i \sqrt{-w} \sqrt{w + 2}\bigr) \\
 & = -\Arg\Bigl(2 - \tfrac{1}{2} (s + s^{-1}) + i \sqrt{\tfrac{1}{2} (s^{-1} + s) - 1} \sqrt{3 - \tfrac{1}{2} (s^{-1} + s)}\Bigr).
}
The above expression belongs to $(-\pi, 0)$ when $s < 1$ and to $(0, \pi)$ when $s > 1$. We conclude that $\ph$ indeed has the desired form:
\formula{
 \ph(s) & = \begin{cases}
  0 \vphantom{\dfrac{\sqrt{s^{-1}}}{s^{-1}}} & \text{if $s < 0$;} \\
  -1 \vphantom{\dfrac{\sqrt{s^{-1}}}{s^{-1}}} & \text{if $0 < s < 3 - 2 \sqrt{2}$;} \\
  -\dfrac{1}{\pi} \arccot \dfrac{4 - s - s^{-1}}{\sqrt{s + s^{-1} - 2} \sqrt{6 - s - s^{-1}}} & \text{if $3 - 2 \sqrt{2} < s < 1$;} \\
  \dfrac{1}{\pi} \arccot \dfrac{4 - s - s^{-1}}{\sqrt{s + s^{-1} - 2} \sqrt{6 - s - s^{-1}}} & \text{if $1 < s < 3 - \sqrt{2}$;} \\
  1 \vphantom{\dfrac{\sqrt{s^{-1}}}{s^{-1}}} & \text{if $s > 3 + 2 \sqrt{2}$.}
 \end{cases}
}

%
%

\subsection*{Acknowledgements}

We thank the anonymous referee for helpful comments.

%
%

%
%

\end{document}